\documentclass{amsart}
\usepackage[utf8]{inputenc}\usepackage{amsmath}
\usepackage{amsfonts}
\usepackage{amssymb}
\usepackage[english]{babel}
\usepackage{amsmath,amsthm}
\usepackage{amsfonts}
\usepackage{fancyhdr}
\usepackage{setspace}
\usepackage{color}
\usepackage[percent]{overpic}
\usepackage{fix-cm}
\usepackage{graphicx}
\usepackage[left=4cm,top=2cm,right=4cm,foot=2cm]{geometry}
\usepackage[all]{xy}
\usepackage{mathrsfs}
\usepackage{mathdots}
\usepackage{amsbsy}
\usepackage[hidelinks]{hyperref}
\usepackage{xcolor}
\hypersetup{
    linktocpage,
    colorlinks,
    linkcolor={magenta!100!black},
    citecolor={cyan!100!black},
    urlcolor={yellow!100!black}
}

\theoremstyle{plain}
\newtheorem{teo}{Theorem}[section]
\newtheorem{lema}[teo]{Lemma}
\newtheorem{prop}[teo]{Proposition}
\newtheorem*{propsn}{Proposition}
\newtheorem{cor}[teo]{Corollary}

\theoremstyle{definition}
\newtheorem{dfn}[teo]{Definition}
\newtheorem{rem}[teo]{Remark}

\theoremstyle{definition}

\theoremstyle{definition}
\newtheorem{hecho}[teo]{Fact}

\theoremstyle{definition}

\theoremstyle{definition}

\theoremstyle{remark}
\newtheorem{cla}[teo]{Claim}

\theoremstyle{remark}
\newtheorem{ex}[teo]{Example}

\newcommand{\liep}{\mathfrak{p}}
\newcommand{\liek}{\mathfrak{k}}
\newcommand{\lieq}{\mathfrak{q}}
\newcommand{\lieh}{\mathfrak{h}}
\newcommand{\lieg}{\mathfrak{g}}
\newcommand{\lieb}{\mathfrak{b}}

\newcommand{\symg}{X_G}

\newcommand{\rr}{\mathbb{R}}
\newcommand{\gr}{\mathsf{Gr}_{d-1}(\mathbb{R}^d)}

\newcommand{\zz}{\mathbb{Z}}

\newcommand{\hh}{\mathbb{H}^2}

\newcommand{\hpq}{\mathbb{H}^{p,q-1}}
\newcommand{\rpq}{\mathbb{R}^{p,q}}
\newcommand{\prpq}{\mathbb{P}(\mathbb{R}^{p,q})}

\newcommand{\hp}{\mathbb{H}^{p}}
\newcommand{\pp}{\mathbb{P}}
\newcommand{\ppc}{\mathbb{P}^{(2)}}
\newcommand{\ppcu}{\mathbb{P}^{(4)}}
\newcommand{\bb}{\mathbb{B}}

\newcommand{\bc}{\begin{center}}
\newcommand{\ec}{\end{center}}

\newcommand{\too}{\longrightarrow}

\newcommand{\ppq}{\perp_{p,q}}

\newcommand{\bg}{\partial_\infty\Gamma}
\newcommand{\bgc}{\partial_\infty^{2}\Gamma}
\newcommand{\bgcu}{\partial_\infty^{4}\Gamma}
\newcommand{\gh}{\Gamma_{\tn{H}}}

\newcommand{\tn}{\textnormal}
\newcommand{\tc}{\textcolor}

\numberwithin{equation}{section}

\newcounter{tmp}

\begin{document}

\title[Counting in special-orthogonal symmetric spaces]{Counting problems for special-orthogonal Anosov representations}
\author{León Carvajales}
\thanks{Research partially funded by CSIC MIA (2017), CSIC Iniciación C068 (2018) and ANR-16-CE40-0025.}

\address[\textbf{León Carvajales}]{\newline Centro de Matemática\newline Universidad de la República\newline Iguá 4225, 11400, Montevideo, Uruguay.}
\email{lcarvajales@cmat.edu.uy}

\maketitle

\begin{abstract}

For positive integers $p$ and $q$ let $G:=\tn{PSO}(p,q)$ be the projective indefinite special-orthogonal group of signature $(p,q)$. We study counting problems in the Riemannian symmetric space $X_G$ of $G$ and in the pseudo-Riemannian hyperbolic space $\hpq$. Let $S\subset X_G$ be a totally geodesic copy of $X_{\tn{PSO}(p,q-1)}$. We look at the orbit of $S$ under the action of a projective Anosov subgroup of $G$. For certain choices of such a geodesic copy we show that the number of points in this orbit which are at distance at most $t$ from $S$ is finite and asymptotic to a purely exponential function as $t$ goes to infinity. We provide an interpretation of this result in $\hpq$, as the asymptotics of the amount of space-like geodesic segments of maximum length $t$ in the orbit of a point.

\end{abstract}

\tableofcontents
\addtocontents{toc}{\setcounter{tocdepth}{1}}

\section{Introduction} \label{sec introduction}
\setcounter{equation}{0}

Let $X$ be a proper non compact metric space and  $o$ be a point in $X$. Given a discrete group $\Delta$ of isometries of $X$, consider the \textit{orbital counting function}

\bc
$N_\Delta(o,t):=\#\lbrace g\in\Delta:\hspace{0,3cm} d_X(o,g\cdot o)\leq t \rbrace$,
\ec

\noindent where $t\geq 0$. The \textit{orbital counting problem} consists on the study of the asymptotic behaviour of $N_\Delta(o,t)$ as $t\too\infty$.

When $X=\rr^2$ and $\Delta=\zz^2$ this is known as the \textit{Gauss circle problem} (see Phillips-Rudnick \cite{PR}). For a negatively curved Hadamard manifold $X$ and $\Delta$ co-compact, this problem was studied by Margulis in his PhD Thesis \cite{Mar}: the author shows a purely exponential asymptotic for $N_\Delta(o,t)$, the exponent being the topological entropy of the geodesic flow of the quotient space $\Delta \backslash X$. Many authors have generalized the work of Margulis to different contexts, see Roblin \cite{Rob} and references therein for a fairly complete picture in the negatively curved setting.

When $X$ is a (not necessarily Riemannian) symmetric space associated to a semisimple Lie group $G$ and $\Delta <G$ is a lattice, these kind of problems were studied notably by Eskin-McMullen \cite{EMcM} and Duke-Rudnick-Sarnak \cite{DRS}. In the non-lattice case but restricted to Riemannian symmetric spaces, one also finds the work of Quint \cite{Qui} and Sambarino \cite{Sam2}. Quint deals with the case in which $\Delta$ is a Schottky group (in the sense of Benoist \cite{Ben4}). Sambarino treats more generally the case of Anosov subgroups (in the full flag variety of $G$) introduced by Labourie \cite{Lab}.

Before stating precise results we discuss in an informal way the problems adressed by this paper. Fix $d:=p+q$ where $p\geq 1$ and $q\geq 2$, and let $\langle\cdot,\cdot\rangle_{p,q}$ be the bilinear symmetric form on $\rr^{d}$ defined by

\bc
$ \langle (x_1,\dots,x_d),(y_1,\dots,y_d)\rangle_{p,q}:= \displaystyle\sum_{i=1}^p x_iy_i - \displaystyle\sum_{i=p+1}^d x_{i}y_{i}$.
\ec

\noindent We denote by $G:=\textnormal{PSO}(p,q)$ the group of projectivized matrices in $\tn{SL}(d,\rr)$ preserving $\langle\cdot,\cdot\rangle_{p,q}$ and by $X_G$ the \textit{Riemannian symmetric space} of $G$, that is, the space of $q$-dimensional subspaces of $\rr^{d}$ on which the form $\langle\cdot,\cdot\rangle_{p,q}$ is negative definite. Let $d_{X_G}$ be the distance in $X_G$ induced by a $G$-invariant Riemannian metric. For closed subsets $A$ and $B$ of $X_G$, set

\bc
$d_{X_G}(A,B):=\inf\lbrace d_{X_G}(a,b):\hspace{0,3cm} a\in A, b\in B\rbrace$.
\ec

\noindent On the other hand, the \textit{pseudo-Riemannian hyperbolic space of signature $(p,q-1)$} is the set

\bc
$\hpq:=\left\lbrace o=[\hat{o}]\in\mathbb{P}(\rr^d): \hspace{0,3cm} \langle \hat{o},\hat{o}\rangle_{p,q}<0  \right\rbrace$,
\ec

\noindent endowed with a $G$-invariant pseudo-Riemannian metric coming from restriction of the form $\langle\cdot,\cdot\rangle_{p,q}$ to tangent spaces.

Let $\Delta$ be a discrete subgroup of $G$ and fix a point $o$ in $\hpq$. In this paper we study counting problems in $X_G$ and in $\hpq$.

\begin{itemize}
\item \textbf{Counting in $X_G$:} Denote by

\bc
$S^o:=\lbrace \tau\in X_G:\hspace{0,3cm} o\subset \tau \rbrace $.
\ec

\noindent It is a totally geodesic sub-manifold of $X_G$ isometric to the Riemannian symmetric space of $\tn{PSO}(p,q-1)$. We define two counting functions in this setting. The first one is

\bc
$N_\Delta(S^o,t):=\#\lbrace g\in\Delta:\hspace{0,3cm} d_{X_G}(S^o,g \cdot S^o)\leq t\rbrace$.
\ec

\noindent For the second one we pick a point $\tau\in S^o$ and define

\bc
$N_\Delta(S^o,\tau,t):=\#\lbrace g\in\Delta:\hspace{0,3cm} d_{X_G}(\tau,g\cdot S^o)\leq t\rbrace$.
\ec

\item \textbf{Counting in $\hpq$:} We provide a geometric interpretation of the function $N_\Delta(S^o,t)$ in $\hpq$. It is the amount of \textit{space-like} geodesic segments\footnote{That is, geodesic segments which are tangent to positive vectors.} of length at most $t$, that connect $o$ with points of $\Delta\cdot o$. The function $N_\Delta(S^o,\tau,t)$ has a geometric interpretation in this setting as well, which is more involved, and that we postpone until Subsection \ref{subsec counting hpq introd}.

\end{itemize}

\begin{rem}
If $q=1$ one has $\hp=X_G=\hpq$ and $o=S^o=\tau$. We have as well the equalities
\bc
$N_\Delta(o,t)=N_\Delta(S^o,t)=N_\Delta(S^o,\tau,t)$
\ec 

\noindent and our results correspond to the classical and well-known counting theorems already quoted.
\begin{flushright}
$\diamond$
\end{flushright}
\end{rem}

In contrast with the counting function $N_\Delta(o,t)$ described at the beginning, the functions $N_\Delta(S^o,t)$ and $N_\Delta(S^o,\tau,t)$ could in general be equal to infinity for large values of $t$. Part of the results that we present here concern the study of conditions for the choice of $o$ (and $\tau$) that guarantee that the new counting functions are real-valued for every $t\geq 0$. Once this is established, one may ask if the \textit{exponential growth rate}

\bc
$\displaystyle\limsup_{t\too\infty}\dfrac{\log N_\Delta(S^o,t)}{t}$
\ec

\noindent is positive, finite and independent on the choice of $o$ (and the analogue questions for $N_\Delta(S^o,\tau,t)$). A more subtle problem is to find an asymptotic for the functions $N_\Delta(S^o,t)$ and $N_\Delta(S^o,\tau,t)$ as $t\too\infty$. The main goal of this paper is to give an answer to this more subtle problem for an interesting class of subgroups $\Delta$: images of word hyperbolic groups under \textit{projective Anosov representations}.

\subsection{Main results in $X_G$} \label{subsec counting XG introd}

In order to formally state our results we need to recall some basic facts concerning (projective) Anosov representations. Anosov representations are (a stable class of) faithful and discrete representations from word hyperbolic groups into semisimple Lie groups that share many geometrical and dynamical features with holonomies of convex co-compact hyperbolic manifolds. They were introduced by Labourie \cite{Lab} in his study of the Hitchin component and further extended to arbitrary word hyperbolic groups by Guichard-Wienhard \cite{GW}. After that, Anosov representations had been object of intensive research in the field of geometric structures on manifolds and their deformation spaces (see for instance the surveys of Kassel \cite{Kas} or Wienhard \cite{Wie} and references therein).

Let $P_1^{p,q}$ be the stabilizer of an isotropic line in $\rr^d$, i.e. a line on which the form $\langle\cdot,\cdot\rangle_{p,q}$ equals zero. Then $P_1^{p,q}$ is a parabolic subgroup of $G$ and the quotient space $\partial\hpq:= G/P_1^{p,q}$, called the \textit{boundary} of $\hpq$, identifies with the set of isotropic lines in $\rr^d$.

Fix a non elementary word hyperbolic group $\Gamma$ and let $\bg$ be its Gromov boundary. Let $\rho:\Gamma\too G$ be a $P_1^{p,q}$-Anosov representation. By definition (see Section \ref{sec anosov}) this means that there exists a continuous equivariant map

\bc
$\xi:\bg\too \partial\hpq$
\ec

\noindent with the following properties:

\begin{itemize}
\item \textbf{Transversality:} Let $\cdot^{\ppq}$ denote the orthogonal complement with respect to the form $\langle\cdot,\cdot\rangle_{p,q}$. Then the map $\eta:=\xi^{\ppq}$ satisfies $\xi(x)\oplus\eta(y)=\rr^d$ for every $x\neq y$ in $\bg$.
\item \textbf{Uniform hyperbolicity:} Some flow associated to $\rho$ satisfies a uniform contraction/dilation property (see \cite{Lab,GW}).
\end{itemize}

When $\rho$ is $P_1^{p,q}$-Anosov all infinite order elements in $\rho(\Gamma)$ are \textit{proximal}. This means that they act on $\pp(\rr^d)$ with a unique attractive fixed line and a unique repelling hyperplane. The \textit{limit set} of $\rho$ is, by definition, the closure  of the set of attractive fixed lines of proximal elements in $\rho(\Gamma)$. It is denoted by $\Lambda_{\rho(\Gamma)}$ and coincides with the image of $\xi$.

Define

\bc

$\pmb{\Omega}_{\rho}:=\lbrace o=[\hat{o}]\in\hpq: \langle \hat{o},\hat{\xi} \rangle_{p,q}\neq 0 \textnormal{ for all }\xi=[\hat{\xi}]\in\Lambda_{\rho(\Gamma)}\rbrace$.
\ec

In the study of discrete groups of projective transformations, it is standard to consider sets similar to $\pmb{\Omega}_\rho$ (see for instance Danciger-Guéritaud-Kassel \cite{DGK1,DGK2} and references therein). Without any further assumption the set $\pmb{\Omega}_\rho$ could be empty. An important class of Anosov representations for which $\pmb{\Omega}_\rho$ is non empty is given by \textit{$\hpq$-convex co-compact subgroups} introduced in \cite{DGK1}. However in our results we do not assume that $\rho$ is $\hpq$-convex co-compact, we only need that $\pmb{\Omega}_\rho\neq\emptyset$ (see Example \ref{ex omegarho no vacio}).

\begin{propsn}[Propositions \ref{prop counting with nu is well defined} and \ref{prop counting with lambdauno is well defined}]
Let $\rho:\Gamma\too G$ be a $P_1^{p,q}$-Anosov representation, a point $o\in\pmb{\Omega}_{\rho}$ and $\tau\in S^o$. Then for every $t\geq 0$ one has\footnote{Even though finiteness of $N_{\rho(\Gamma)}(S^o,\tau,t)$ follows directly from finiteness of $N_{\rho(\Gamma)}(S^o,t)$, in our proof we first show $N_{\rho(\Gamma)}(S^o,\tau,t)<\infty$ and use it to prove $N_{\rho(\Gamma)}(S^o,t)<\infty$.}

\bc
$N_{\rho(\Gamma)}(S^o,\tau,t)<\infty$ and $N_{\rho(\Gamma)}(S^o,t)<\infty$.
\ec

\end{propsn}

The main results of this paper in the Riemannian context are Theorems \ref{teorema A} and \ref{teorema B}. The notation $f(t) \sim g(t)$ stands for

\bc
$\displaystyle\lim_{t\too\infty}\frac{f(t)}{g(t)}= 1$.
\ec

\begingroup
\setcounter{tmp}{\value{teo}}% store current value of theorem counter
\setcounter{teo}{0} %assign desired value to theorem counter
\renewcommand\theteo{\Alph{teo}}% locally redefine the representation of the theorem counter

\begin{teo} \label{teorema A}
Let $\rho:\Gamma\too G$ be a $P_1^{p,q}$-Anosov representation and $o\in\pmb{\Omega}_{\rho}$. There exist positive constants $h=h_\rho$ and $M=M_{\rho,o}$ such that

\bc
$N_{\rho(\Gamma)}(S^o,t)\sim \dfrac{e^{ht}}{M}$.
\ec

\end{teo}

\begin{teo} \label{teorema B}
Let $\rho:\Gamma\too G$ be a $P_1^{p,q}$-Anosov representation, a point $o\in\pmb{\Omega}_{\rho}$ and $\tau\in S^o$. There exist positive constants $h=h_\rho$ and $M'=M'_{\rho,\tau}$ such that

\bc
$N_{\rho(\Gamma)}(S^o,\tau,t)\sim \dfrac{e^{ht}}{M'}$.
\ec

\end{teo}

\endgroup

The constant $h$ is the same in both Theorems \ref{teorema A} and \ref{teorema B} and it is independent on the choice of $o$ in $\pmb{\Omega}_\rho$ (and $\tau$ in $S^o$). It coincides with the topological entropy of the \textit{geodesic flow} $\phi^\rho$ of $\rho$, introduced by Bridgeman-Canary-Labourie-Sambarino \cite{BCLS}, and can be computed as

\bc
$h=\displaystyle\limsup_{t\too\infty}\dfrac{\log\#\lbrace[\gamma]\in[\Gamma]:\hspace{0,3cm}\lambda_1(\rho(\gamma))\leq t\rbrace}{t}$.
\ec

\noindent Here $[\gamma]$ denotes the conjugacy class of $\gamma$ and $\lambda_1(\rho(\gamma))$ denotes the logarithm of the spectral radius of $\rho(\gamma)$. The constants $M$ and $M'$ are related to the total mass of specific measures in the Bowen-Margulis measure class of $\phi^\rho$ (recall that the Bowen-Margulis measure class is the homothety class of measures maximizing entropy of $\phi^\rho$).

Since the work of Margulis \cite{Mar}, in order to obtain a counting result one usually studies the ergodic properties of a well chosen dynamical system. In order to find a dynamical system adapted to Theorem \ref{teorema A} we introduce a decomposition of a specific subset of $G$, analogue to the Cartan Decomposition, but replacing \textit{the} maximal compact subgroup of $G$ by $\tn{PSO}(p,q-1)$ and \textit{the} Cartan subspace by a smaller abelian subalgebra (Subsection \ref{subsec HexpliebH}). For Theorem \ref{teorema B} we use the more studied \textit{polar decomposition} of $G$ (Subsection \ref{subsec KexpliebH}).

\subsubsection{\tn{\textbf{Relation with the work of Oh-Shah}}} \label{subsub ohshah}

Motivated by the study of \textit{Apollonian circle packings} on the Riemann sphere, Oh-Shah \cite{OS} studied counting problems similar to ours. Indeed, let $p=1$ and $q=3$. Then $\mathbb{H}^{1,2}$ identifies with the space of circles of the Riemann sphere or, equivalently, the space of totally geodesic isometric copies of $\mathbb{H}^2$ inside $\mathbb{H}^3$. In \cite[Theorem 1.5]{OS} the cited authors prove that for a well-chosen $S^o\cong \hh\subset\mathbb{H}^3$ and any point $\tau\in\mathbb{H}^3$ one has

\bc
$\#\lbrace g\in\Delta:\hspace{0,3cm} d_{\mathbb{H}^3}(\tau,g\cdot S^o)\leq t  \rbrace\sim M^{-1}e^{ht}$.
\ec

\noindent Hence Theorem \ref{teorema B} can be interpreted as a higher rank generalization of this result. We note however that, for $p=1$ and $q=3$, our results only concern convex co-compact groups, while Oh-Shah's Theorem applies to a wider class of geometrically finite Kleinian subgroups. A slightly different counting theorem in $\mathbb{H}^{1,2}$ was obtained by the cited authors in \cite{OS3}. Effective versions of Oh-Shah's results (i.e. with an error term) have been obtained by Lee-Oh \cite{LO} and Mohammadi-Oh \cite{MO}. Our Theorem \ref{teorema A} seems to be new even in this setting.

The approach by Oh-Shah is similar to the one of Eskin-McMullen \cite{EMcM}: they study the equidistribution, with respect to certain measures, of the orthogonal translates of $S^o$ under the geodesic flow of $\Delta\backslash\mathbb{H}^3$ (see Oh-Shah \cite{OS2} for precisions). Here we use different techniques. We follow the approach by Sambarino \cite{Sam} and construct a dynamical system on a compact space that contains the required geometric information.

\subsection{Interpretation in $\hpq$} \label{subsec counting hpq introd}

Another part of our contributions concern geometric interpretations of Theorems \ref{teorema A} and \ref{teorema B} in $\hpq$. We now state these interpretations.

Geodesics in $\hpq$ are intersections of projectivized $2$-dimensional subspaces of $\rr^d$ with $\hpq$ and they are classified in three types, depending on the sign of the form $\langle\cdot,\cdot\rangle_{p,q}$ on its tangent vectors (see Subsection \ref{subsub geod hpq}). We are mainly interested in \textit{space-like geodesics}, i.e. geodesics associated to planes on which the form $\langle\cdot,\cdot\rangle_{p,q}$ has signature $(1,1)$. Let $o,o'\in\hpq$ be two points joined by a space-like geodesic and let $\ell_{o,o'}$ be the length of this geodesic segment (see Subsection \ref{subsub lengths of spacelike geod}). We denote by $\mathscr{C}^>_o$ the set of points of $\hpq$ that can be joined to $o$ by a space-like geodesic and we set

\bc
$\mathscr{C}^>_{o,G}:=\lbrace g\in G:\hspace{0,3cm} g\cdot o\in\mathscr{C}^>_o\rbrace$.
\ec

\begin{propsn}[Proposition \ref{prop ell dXG y vertboverto}]
Let $o\in\hpq$ and $g\in \mathscr{C}^>_{o,G}$. Then 

\bc
$\ell_{o,g\cdot o}=d_{X_G}(S^o,g\cdot S^o)$.
\ec
\end{propsn}

In Corollary \ref{cor gammao in cowmayor} we prove that given a $P_1^{p,q}$-Anosov representation $\rho:\Gamma\too G$ and $o$ in $\pmb{\Omega}_\rho$, then apart from possibly finitely many exceptions $\gamma$ in $\Gamma$ one has $\rho(\gamma)\in \mathscr{C}^>_{o,G}$. By Proposition \ref{prop counting with lambdauno is well defined} we have

\bc
$\#\lbrace \gamma\in\Gamma: \hspace{0,3cm} \rho(\gamma)\in \mathscr{C}^>_{o,G} \tn{ and } \ell_{o,\rho(\gamma)\cdot o}\leq t\rbrace<\infty$
\ec

\noindent for every positive $t$. Moreover, Theorem \ref{teorema A} implies that this function is asymptotic to $M^{-1}e^{ht}$ as $t\too\infty$.

In order to state the corresponding geometric interpretation of Theorem \ref{teorema B} we follow Kassel-Kobayashi \cite[p.151]{KK}. Let $o\in\hpq$ and $\tau\in S^o$. Then

\bc
$\mathbb{H}^p_\tau:=(o\oplus\tau^{\ppq})\cap\hpq$
\ec

\noindent is a space-like totally geodesic copy of $\hp$ passing through $o$. Let $K^\tau$ be the (maximal compact) subgroup of $G$ stabilizing $\tau$. As we shall see, for every $g$ in $G$ the point $g\cdot o$ lies in the $K^\tau$-orbit of a point $o_g$ in $\hp_\tau$. The counterpart of Theorem \ref{teorema B} in $\hpq$ is provided by the following proposition.

\begin{propsn}[Proposition \ref{prop interpretation of btau in symg}]
For every $g$ in $G$ one has

\bc
$\ell_{o,o_g}=d_{X_G}(\tau,g\cdot S^o)$.
\ec

\end{propsn}

\subsubsection{\tn{\textbf{Relation with the work of Glorieux-Monclair and Kassel-Kobayashi}}}\label{subsub GM and KK in introd}

Glorieux-Monclair \cite{GM} introduced an orbital counting function for $\hpq$-convex co-compact representations that differs from

\bc
$t\mapsto\#\lbrace \gamma\in\Gamma: \hspace{0,3cm} \rho(\gamma)\in \mathscr{C}^>_{o,G} \tn{ and } \ell_{o,\rho(\gamma)\cdot o}\leq t\rbrace$
\ec

\noindent by a constant. Indeed, they define an \textit{$\hpq$-distance}

\bc
$d_{\hpq}(o,o'):=\left\{\begin{array}{cc} \ell_{o,o'} \hspace{0,1cm}\textrm{ if } o'\in\mathscr{C}^>_o \\ 0  \hspace{0,6cm}\textrm{ otherwise }\end{array}\right.$,
\ec 

\noindent and show that it satisfies a version of the triangle inequality in the convex hull of the limit set of $\rho$. This is used to prove that the exponential growth rate of the counting function

\bc
$t\mapsto\#\lbrace \gamma\in\Gamma: \hspace{0,3cm}  d_{\hpq}(o,\rho(\gamma)\cdot o)\leq t\rbrace$
\ec

\noindent is independent on the choice of the basepoint $o$. The authors interpret this exponential rate as a \textit{pseudo-Riemannian Hausdorff dimension} of the limit set of $\rho$, with the purpose of finding upper bounds for this number (\cite[Theorem 1.2]{GM}). A consequence of Theorem \ref{teorema A} and Proposition \ref{prop ell dXG y vertboverto} (see Remarks \ref{rem crit exponent coindes with the one of GM} and \ref{rem crit exponent coincides with the entropy}) is that this rate coincides with the topological entropy $h$ of $\phi^\rho$.

On the other hand, as we shall see in Section \ref{sec generalized cartan} the number $\ell_{o,o_g}$ is related to the \textit{polar projection} of $g$ and therefore Theorem \ref{teorema B} addresses the problems treated by Kassel-Kobayashi in \cite[Section 4]{KK}. In \cite{KK} the authors study the orbital counting function of Theorem \ref{teorema B} for \textit{sharp} subgroups of a real reductive symmetric space (see \cite[Section 4]{KK}). Kassel-Kobayashi obtain some estimates on the growth of this function, but no precise asymptotic is established.

The method of \cite{GM} is based on pseudo-Riemannian geometry: they construct analogues of Busemann functions, Gromov products and Patterson-Sullivan densities in $\hpq$ using this viewpoint. Our approach is inspired by \cite{KK} and has Lie-theoretic flavor: we study linear algebraic interpretations of the geometric quantities involved in the definition of the counting functions. This allows us to establish finiteness of these functions, to make a link between the different symmetric spaces and to apply Ledrappier's \cite{Led} framework to our setting.

\subsection{Outline of the proof} \label{subsec outline}

There are three major steps in the proof of Theorems \ref{teorema A} and \ref{teorema B}.

\subsubsection*{\textnormal{\textbf{First step}}} As we said, we interpret the geometric quantities involved in Theorems \ref{teorema A} and \ref{teorema B} as linear algebraic quantities.

Let us be more precise. Fix $o \in \hpq$ and denote by $H^o$ the stabilizer in $G$ of this point. If we consider the symmetry of $\rr^d$ given by $J^o:=\tn{id}_o\oplus\left(-\tn{id}_{o^{\ppq}}\right)$, we have that $H^o$ equals the fixed point set of the involution
\bc
$\sigma^o:g\mapsto J^ogJ^o$
\ec

\noindent of $G$ (see Subsection \ref{subsubsec struc sym hpq}). This identifies the tangent space at $o$ of $\hpq$ with the subspace of $\mathfrak{so}(p,q)$ defined by $\lieq^o:=\lbrace d\sigma^o=-1\rbrace$. In Propositions \ref{prop ell dXG y vertboverto} and \ref{prop linear alg interpr of bo} we prove that for every $g\in \mathscr{C}^>_{o,G}$ one has
\begin{equation}\label{eq igualddad distancia con vap}
d_{X_G}(S^o,g\cdot S^o)=\frac{1}{2}\lambda_1(J^ogJ^og^{-1}).
\end{equation}

\noindent The main ingredient in the proof of equality (\ref{eq igualddad distancia con vap}) is the following version of the classical Cartan Decomposition of $G$.

\begin{propsn}[Proposition \ref{prop HBH}]
Let $o\in\hpq$ and $\lieb^+\subset\lieq^o$ be a ray such that $\exp(\lieb^+)\cdot o$ is space-like. Given $g\in \mathscr{C}^>_{o,G}$ there exists $h,h'\in H^o$ and a unique $X\in\lieb^+$ such that

\bc
$g=h\exp(X)h'$.
\ec

\end{propsn} 

On the other hand, the linear algebraic interpretation of the quantity $d_{X_G}(\tau,g\cdot S^o)$ is the following: the choice of $\tau$ induces a norm $\Vert\cdot\Vert_\tau$ on $\rr^d$ invariant under the action of $K^\tau$. We show in Propositions \ref{prop interpretation of btau in symg} and \ref{prop computing nu} that for every $g\in G$ the following equality holds
\begin{equation}\label{eq igualddad distancia con vasing}
d_{X_G}(\tau,g\cdot S^o)=\frac{1}{2}\log\Vert J^ogJ^og^{-1}\Vert_\tau.
\end{equation}
\noindent Once again the proof of this equality relies on a generalization of Cartan Decomposition (see Schlichtkrull \cite[Chapter 7]{Sch}): every $g\in G$ can be written as 

\bc
$g=k\exp(X)h$
\ec

\noindent for some $k\in K^{\tau}$, $h\in H^o$ and a unique $X\in\lieb^+$.

\subsubsection*{\textnormal{\textbf{Second step}}} In order to simplify the exposition we assume that $\Gamma$ is torsion free. In this case every $\gamma\neq 1$ in $\Gamma$ has a unique attractive (resp. repelling) fixed point in $\bg$, denoted by $\gamma_+$ (resp. $\gamma_-$). Consider $\rho:\Gamma\too G$ a $P_1^{p,q}$-Anosov representation. The key feature of choosing $o$ in $\pmb{\Omega}_\rho$ is that it guarantees some \textit{transversality condition} for the proximal matrices $J^o\rho(\gamma)J^o$ and $\rho(\gamma^{-1})$ and this allows to estimate the quantities (\ref{eq igualddad distancia con vap}) and (\ref{eq igualddad distancia con vasing}) in terms of the spectral radius of $\rho(\gamma)$.

More precisely, we will see in Proposition \ref{prop fijos de Jo en borde} that 
\begin{equation} \label{eq omegarho en introduccion}
\pmb{\Omega}_\rho=\lbrace o\in\hpq:\hspace{0,3cm} J^o\cdot\xi(x)\notin\eta(x) \textnormal{ for all } x\in\bg \rbrace.
\end{equation}

\noindent Fix $o\in\pmb{\Omega}_\rho$ and a distance $d$ in $\pp(\rr^d)$ induced by the choice of an inner product in $\rr^d$. By compactness of $\bg$ there exists a positive constant $r$ such that

\bc
$d(J^o\cdot\xi(x),\eta(x))\geq r$
\ec

\noindent holds for every $x\in\bg$ (here $d(J^o\cdot\xi(x),\eta(x))$ is the minimal distance between $J^o\cdot\xi(x)$ and the lines included in $\eta(x)$). Further, if $\gamma_+$ is uniformly far from $\gamma_-$, with respect to some visual distance in $\bg$, then $\xi(\gamma_+)$ (resp. $\xi(\gamma_-)$) is uniformly far from $\eta(\gamma_-)$ (resp. $\eta(\gamma_+)$). In Lemma \ref{lema jrhojrho prox dos} we combine all these facts with Benoist's work \cite{Ben1} to conclude that  the product $J^o\rho(\gamma)J^o\rho(\gamma^{-1})$ is proximal. Moreover, we obtain a comparison between the quantity (\ref{eq igualddad distancia con vap}) (resp. (\ref{eq igualddad distancia con vasing})) and

\bc
$\lambda_1(\rho(\gamma))$
\ec

\noindent with very precise control on the error made in this comparison.

\subsubsection*{\tn{\textbf{Third step}}} 

We apply Sambarino's outline \cite{Sam} to our particular context\footnote{The results in \cite{Sam} are proved for fundamental groups of closed negatively curved manifolds. However, all the results obtained there remain valid when $\Gamma$ is an arbitrary word hyperbolic group admitting an Anosov representation. This is explained in detail in Appendix \ref{appendix distribution utaugamma y uogamma}.}. To a Hölder cocycle $c$ on $\bg$ the author associates a Hölder reparametrization $\psi_t^c$ of the geodesic flow of $\Gamma$. Recall that a \textit{Hölder cocycle} is a map $c:\Gamma\times\bg\too\rr$ satisfying

\bc
$c(\gamma_0\gamma_1,x)=c(\gamma_0,\gamma_1\cdot x)+c(\gamma_1,x)$
\ec

\noindent for every $\gamma_0,\gamma_1$ in $\Gamma$ and $x\in\bg$ and such that the map $c(\gamma_0,\cdot)$ is Hölder (with the same exponent for every $\gamma_0$). The cocycle $c'$ is said to be \textit{cohomologous} to $c$ if there exists a Hölder continuous function $U:\bg\too\rr$ such that for every $\gamma$ in $\Gamma$ and $x$ in $\bg$ one has

\bc
$c(\gamma,x)-c'(\gamma,x)=U(\gamma\cdot x)-U(x)$.
\ec

\noindent In that case $\psi_t^c$ is conjugate to $\psi_t^{c'}$ (see \cite[Section 3]{Sam}). By considering a Markov coding and applying Parry-Pollicott's Prime Orbit Theorem \cite{PP}, Sambarino obtains an asymptotic for the number of periodic orbits of $\psi_t^c$ of period less than or equal to $ t$ (see \cite[Corollary 4.1]{Sam}). Obviously this is a purely dynamical result, i.e. changing $\psi_t^c$ in its conjugacy class does not affect the asymptotics.

However our problem is more subtle: one must find a particular cocycle, with some geometric meaning, and not just \textit{any} cocycle in the given cohomology class. Indeed, the cocycles that we consider to prove Theorems \ref{teorema A} and \ref{teorema B} are cohomologous, but only the specific choices in such a cohomology class yield the respective results.

Let us briefly sketch the proof of Theorem \ref{teorema A} (Theorem \ref{teorema B} is proved in a similar way). Fix $o\in\pmb{\Omega}_\rho$ and consider

\bc
$c_o:\Gamma\times\bg\too\rr: \hspace{0,3cm} c_o(\gamma,x):=\dfrac{1}{2}\log\left\vert\dfrac{\langle\rho(\gamma)\cdot v_x,J^o\rho(\gamma)\cdot v_x\rangle_{p,q}}{\langle v_x,J^o\cdot v_x\rangle_{p,q}}\right\vert$
\ec

\noindent where $v_{x}\neq 0$ is any vector in $\xi(x)$\footnote{When $q=1$ this coincides with the \textit{Busemann cocycle} of $\mathbb{H}^{p}$, i.e. $c_o(\gamma,x)=\beta_{\xi(x)}(\rho(\gamma^{-1})\cdot o,o)$ where $\beta_\cdot(\cdot,\cdot):\partial\hp\times\hp\times\hp\too\rr$ is the Busemann function.}. This is a well-defined function thanks to (\ref{eq omegarho en introduccion}) and it is a Hölder cocycle.

Let $\bgc$ be the set of pairs of distinct points in $\bg$ and consider the action of $\Gamma$ on $\bgc\times\rr$ given by

\bc
$\gamma\cdot (x,y,s):=(\gamma \cdot x,\gamma\cdot y, s-c_o(\gamma,y))$.
\ec

\noindent We denote by $\tn{U}_o\Gamma$ the quotient space. The \textit{translation flow} on $\bgc\times\rr$ given by

\bc
$\psi_t(x,y,s):=(x,y,s-t)$
\ec

\noindent descends to a flow $\psi_t=\psi_t^o$ on $\tn{U}_o\Gamma$. As Sambarino shows in \cite[Theorem 3.2(1)]{Sam} (see also Lemma \ref{lema conj urhogamma y uogamma}) the flow $\psi_t$ is conjugate to a Hölder reparametrization of the geodesic flow of $\Gamma$ introduced by Gromov \cite{Gro}. We will show (see Lemma \ref{lema conj urhogamma y uogamma}) that periodic orbits of $\psi_t$ are parametrized by conjugacy classes of \textit{primitive} elements in $\Gamma$, i.e. elements which cannot be written as a power of another element. If $\gamma$ is primitive, the corresponding period is given by

\bc
$\ell_{c_o}(\gamma):=\lambda_1(\rho(\gamma))$.
\ec

We show the following property concerning spectral radii in a projective Anosov representation.

\begin{propsn}[Proposition \ref{prop geod flow is weak mixing}]

Let $\rho$ be a projective Anosov representation of $\Gamma$. Then the set $\lbrace\lambda_1(\rho(\gamma))\rbrace_{\gamma\in\Gamma}$ spans a non discrete subgroup of $\rr$.

\end{propsn}

Denote by $h$ the topological entropy of $\psi_t$. The probability of maximal entropy of $\psi_t$ can be constructed as follows: define the \textit{Gromov product}

\bc
$[\cdot,\cdot]_o:\bgc\too\rr: \hspace{0,3cm} [x,y]_o:=-\dfrac{1}{2}\log\left\vert  \dfrac{\langle v_x,J^o\cdot v_x\rangle_{p,q}\langle v_y,J^o\cdot v_y\rangle_{p,q}}{\langle v_x,v_y\rangle_{p,q}\langle v_y,v_x\rangle_{p,q}}\right\vert$.
\ec

\noindent This function is well-defined thanks to (\ref{eq omegarho en introduccion}) and transversality of $\xi$ and $\eta$. One can prove that

\bc
$[\gamma \cdot x,\gamma\cdot y]_o - [x,y]_o=-(c_o(\gamma,x)+c_o(\gamma,y))$
\ec

\noindent holds for every $\gamma$ in $\Gamma$ and $(x,y)\in\bgc$. Let $\mu_o$ be a \textit{Patterson-Sullivan probability} associated to $c_{o}$, that is, $\mu_o$ is a probability on $\bg$ that satisfies

\bc
$\dfrac{d\gamma_*\mu_o}{d\mu_o}(x)=e^{-hc_{o}(\gamma^{-1},x)}$
\ec

\noindent for every $\gamma\in\Gamma$\footnote{Recall that if $f:X\too Y$ is a map and $m$ is a measure on $X$ then $f_*(m)$ denotes the measure on $Y$ defined by $A\mapsto m(f^{-1}(A))$.}. For the existence of such a probability see Subsection \ref{subsub PS}. The measure

\bc
$e^{-h[\cdot,\cdot]_o}\mu_o\otimes\mu_o\otimes dt$
\ec

\noindent on $\bgc\times\rr$ is $\Gamma$-invariant. It induces on the quotient $\tn{U}_o\Gamma$ the measure of maximal entropy of $\psi_t$, which is unique up to scaling (see \cite[Theorem 3.2(2)]{Sam} or Proposition \ref{prop product is of maximal entropy}).

Denote by $C_c^*(\bgc)$ the dual of the space of compactly supported real continuous functions on $\bgc$ equipped with the weak-star topology. For $x$ in $\bg$ let $\delta_x$ be the Dirac mass at $x$. Inspired by the work of Roblin \cite{Rob}, Sambarino \cite[Proposition 4.3]{Sam} shows
 
\bc
$Me^{-ht}\displaystyle\sum_{\gamma\in\Gamma, \ell_{c_o}(\gamma)\leq t} \delta_{\gamma_-}\otimes\delta_{\gamma_+}\too e^{-h[\cdot,\cdot]_o}\mu_o\otimes\mu_o$
\ec

\noindent on $C_c^*(\bgc)$ as $t\too\infty$ (for a proof in our context see Proposition \ref{prop distribution of periodic orbits}). The constant $M=M_{\rho,o}>0$ equals the product of $h$ with the total mass of $e^{-h[\cdot,\cdot]_o}\mu_o\otimes\mu_o\otimes dt$ on the quotient space $\tn{U}_o\Gamma$.

As we show in Lemma \ref{lema computing gromov on gammapm o}, the number $[\gamma_-,\gamma_+]_o$ is the precise error term in the comparison between $\ell_{c_o}(\gamma)$ and $\frac{1}{2}\lambda_1(J^o\rho(\gamma)J^o\rho(\gamma^{-1}))=d_{X_G}(S^o,\rho(\gamma)\cdot S^o)$ provided by Benoist's Theorem \ref{teo benoist}. This is the geometric step: we replace the period $\ell_{c_o}(\gamma)$ by the number $d_{X_G}(S^o,\rho(\gamma)\cdot S^o)$ in the previous sum, using the Gromov product.

\begin{propsn}[Proposition \ref{prop distribution on bg for length}]
Let $\Gamma$ be a torsion free word hyperbolic group, $\rho:\Gamma\too G$ be a $P_1^{p,q}$-Anosov representation and $o\in\pmb{\Omega}_{\rho}$. Then

\bc
$M e^{-ht}\displaystyle\sum_{\gamma\in\Gamma, d_{X_G}(S^o,\rho(\gamma)\cdot S^o) \leq t} \delta_{\gamma_-}\otimes\delta_{\gamma_+}\too \mu_o\otimes\mu_o$
\ec

\noindent on $C^*(\bg\times\bg)$ as $t\too\infty$.
\end{propsn}

The proof of Proposition \ref{prop distribution on bg for length} follows line by line the proof of \cite[Theorem 6.5]{Sam}, which is again inspired by Roblin's work \cite{Rob}.

It turns out that the previous proposition can be used to deduce Theorem \ref{teorema A} in the general case, that is, if we admit torsion elements in $\Gamma$.

\begin{propsn}[Proposition \ref{prop distribution on bg for length with torsion}]
Let $\rho:\Gamma\too G$ be a $P_1^{p,q}$-Anosov representation and $o\in\pmb{\Omega}_{\rho}$. Then

\bc
$M e^{-ht}\displaystyle\sum_{\gamma\in\Gamma, d_{X_G}(S^o,\rho(\gamma)\cdot S^o) \leq t} \delta_{\rho(\gamma^{-1})\cdot o^{\ppq}}\otimes\delta_{\rho(\gamma)\cdot o}\too \eta_{*}(\mu_o)\otimes\xi_*(\mu_o)$
\ec

\noindent on $C^*(\pp((\rr^d)^*)\times\pp(\rr^d))$ as $t\too\infty$.
\end{propsn}

\subsection{Organization of the paper}

In Section \ref{sec symmetric spaces} we recall basic facts on the symmetric spaces $X_G$ and $\hpq$. Of particular importance is Subsection \ref{subsub endign of spacelike}, which is devoted to the study of end points of space-like geodesics passing through our preferred point $o\in\hpq$. We give several characterizations of this set that will allow us to understand $\pmb{\Omega}_\rho$ in different ways, all of them used indistinctly in Sections \ref{sec the set omegarho}, \ref{section distrib wrt bo} and \ref{section distrib wrt btau}. In Section \ref{sec generalized cartan} we study the geometric quantities involved in Theorems \ref{teorema A} and \ref{teorema B}. Equalities (\ref{eq igualddad distancia con vap}) and (\ref{eq igualddad distancia con vasing}) are proven respectively in Subsections \ref{subsec HexpliebH} and \ref{subsec KexpliebH}. In Section \ref{sec proximality} we recall Benoist's results on products of proximal matrices and Section \ref{sec anosov} is devoted to reminders on Anosov representations. In Section \ref{sec the set omegarho} we define the set $\pmb{\Omega}_\rho$ and study the action of $\Gamma$ on this set. We show in particular that the orbital counting functions involved in Theorems \ref{teorema A} and \ref{teorema B} are well-defined (Proposition \ref{prop counting with lambdauno is well defined} and Proposition \ref{prop counting with nu is well defined}). We also obtain some estimates for the spectral radius and operator norm of elements $J^o\rho(\gamma)J^o\rho(\gamma^{-1})$ which are of major importance (c.f. Lemma \ref{lema jrhojrho prox dos}). In Section \ref{section distrib wrt bo} (resp. Section \ref{section distrib wrt btau}) we prove Theorem \ref{teorema A} (resp. Theorem \ref{teorema B}). Finally, in Appendix \ref{appendix distribution utaugamma y uogamma} we explain how to adapt the results of \cite{Sam} to the context of arbitrary word hyperbolic groups admitting an Anosov representation.

\subsection*{Acknowledgements}

These problems were proposed to me by Rafael Potrie and Andrés Sambarino. Without their guidance, their support and the (many) helpful discussions this work would not have been possible. I am extremely grateful for this.

The author also acknowledges Olivier Glorieux, Tal Horesh and Fanny Kassel for several enlightening discussions and comments.

Finally, I would like to thank the referee of this article for careful reading and useful suggestions.

\section{Two symmetric spaces associated to $\textnormal{PSO}(p,q)$} \label{sec symmetric spaces}
\setcounter{equation}{0}

Fix two integers $p,q\geq 1$ and let $d:=p+q$. We assume $d>2$. Denote by $\rpq$ the vector space $\rr ^d$ endowed with the quadratic form

\bc
$ \langle (x_1,\dots,x_d),(y_1,\dots,y_d)\rangle_{p,q}:= \displaystyle\sum_{i=1}^p x_iy_i - \displaystyle\sum_{i=p+1}^d x_{i}y_{i}$.
\ec

\noindent From now on we denote by $G:=\textnormal{PSO}(p,q)$ the subgroup of $\textnormal{PSL}(d,\rr)$ consisting on elements whose lifts to $\tn{SL}(d,\rr)$ preserve the form $ \langle \cdot,\cdot\rangle_{p,q}$.

For a subspace $\pi$ of $\rr^d$ we denote by $\pi^{\ppq}$ its orthogonal complement with respect to  $\langle \cdot,\cdot\rangle_{p,q}$, i.e. 

\bc
$\pi^{\ppq}:=\lbrace x\in\rr^d: \hspace{0,3cm} \langle x,y\rangle_{p,q}=0 \tn{ for all } y\in \pi\rbrace$.
\ec

Let $\mathfrak{g}:=\mathfrak{so}(p,q)$ be the Lie algebra of $G$. If $\cdot^t$ denotes the \textit{usual} transpose operator one has that $\lieg$ equals the set of matrices of the form

\bc
$\left(\begin{matrix}
X_1 & X_2\\
X_2^t & X_3
\end{matrix}\right)$
\ec

\noindent where $X_1$ is of size $p\times p$, $X_3$ is of size $q\times q$ and both are skew-symmetric with respect to $\cdot^t$. The \textit{Killing form} of $G$ is the symmetric bilinear form $\kappa$ on $\lieg$ defined by

\bc
$\kappa(X,Y):=\tn{tr}(\tn{ad}_X\circ\tn{ad}_Y)$,
\ec

\noindent where $\tn{ad}:\lieg\too\tn{End}(\lieg)$ is the adjoint representation.  It can be seen that the following equality holds:

\bc
$\kappa(X,Y)=(d-2)\tn{tr}(XY)$
\ec

\noindent  (see Helgason {\cite[p.180 \& p.189]{Hel}}).

\subsection{The Riemannian symmetric space $X_G$} \label{subsec XG}

A \textit{Cartan involution} of $G$ is an involutive automorphism $\tau:G\too G$ such that the bilinear form

\bc
$(X,Y)\mapsto -\kappa(X,d\tau(Y))$
\ec

\noindent is positive definite. The fixed point set $K^\tau$ of such an involution is a maximal compact subgroup of $G$ (see Knapp \cite[Theorem 6.31]{Kna}). The \textit{Riemannian symmetric space} of $G$ is the set consisting on Cartan involutions of $G$. It is denoted by $X_G$ and it is equipped with a natural action of $G$ which is transitive (c.f. \cite[Corollary 6.19]{Kna}). The stabilizer of $\tau$ is $K^\tau$, thus

\bc
$G/K^\tau\cong X_G$.
\ec

\begin{rem}\label{rem Xg space of qplanes}
The space $X_G$ can be identified with the space of $q$-dimensional subspaces of $\rr^d$ on which the form $\langle\cdot,\cdot\rangle_{p,q}$ is negative definite. Explicitly, to a $q$-dimensional negative definite subspace $\pi$ one associates the Cartan involution of $G$ determined by the inner product of $\rr^d$ which equals $-\langle\cdot,\cdot\rangle_{p,q}$ (resp. $\langle\cdot,\cdot\rangle_{p,q}$) on $\pi$ (resp. $\pi^{\ppq}$) and for which $\pi$ and $\pi^{\ppq}$ are orthogonal.
\begin{flushright}
$\diamond$
\end{flushright}
\end{rem}

The choice of a point $\tau$ in $X_G$ determines a \textit{Cartan decomposition}

\bc
$\lieg=\liep^\tau\oplus\liek^\tau$
\ec

\noindent where $\liep^{\tau}:=\lbrace d\tau=-1\rbrace$ and $\liek^{\tau}:=\lbrace d\tau=1\rbrace$. The group $K^\tau$ is tangent to $\liek^\tau$ and one has a $G$-equivariant identification 
\begin{equation} \label{eq liep es el tangente}
\liep^\tau\cong T_\tau X_G
\end{equation}

\noindent given by $X\mapsto \left. \frac{d}{dt}\right\vert_0 \exp{(tX)}\cdot\tau$ (see \cite[Theorem 3.3 of Ch. IV]{Hel}).

\begin{ex} \label{ex explicit cartan involution}

Consider the involution of $G$ defined by $\tau(g):=(g^{-1})^t$. One sees that $\tau\in X_G$ and $\liep^{\tau}$ (resp. $\liek^{\tau}$) is the set of symmetric matrices (resp. skew-symmetric matrices) in $\mathfrak{so}(p,q)$. Moreover $K^{\tau}$ is the subgroup $\textnormal{PS}(\textnormal{O}(p)\times \textnormal{O}(q))$.
\begin{flushright}
$\diamond$
\end{flushright}
\end{ex}

The Killing form $\kappa$ is positive definite (resp. negative definite) on $\liep^\tau$ (resp. $\liek^\tau$). Thanks to (\ref{eq liep es el tangente}) any positive multiple of $\kappa$ induces a $G$-invariant Riemannian metric on $X_G$. It is well-known (see \cite[Theorem 4.2 of Ch. IV]{Hel}) that $X_G$ equipped with any of these metrics is a symmetric space which is non-positively curved.

We already mentioned that in this paper we study counting problems not only in $X_G$ but also in $\hpq$. In the next section we construct $\hpq$, whose metric is induced by the form $\langle\cdot,\cdot\rangle_{p,q}$. However, we will see that the Killing form induces as well a $G$-invariant metric on $\hpq$. These two metrics differ by the scaling factor $(2(d-2))^{-1}$ (see Remark \ref{rem form on qo} for further precisions). Since we want a simultaneous treatment of the geometry of the spaces $X_G$ and $\hpq$, we fix the following normalization for the metric on $X_G$:
\begin{equation} \label{eq distance in XG and killing}
d_{X_G}(\tau,\exp(X)\cdot\tau):=\left(\dfrac{1}{2(d-2)}\kappa(X,X)\right)^{\frac{1}{2}}
\end{equation}

\noindent for all $\tau\in X_G$ and all $X\in\liep^\tau$.

\subsection{The pseudo-Riemannian hyperbolic space $\hpq$} \label{subsec hpq}

Let

\bc
$\hat{\mathbb{H}}^{p,q-1}:=\lbrace \hat{o}\in\rpq: \hspace{0,3cm} \langle \hat{o},\hat{o}\rangle_{p,q}=-1 \rbrace$
\ec

\noindent endowed with the restriction of the form $\langle \cdot,\cdot\rangle_{p,q}$ to tangent spaces. This metric induces on

\bc
$\hpq:=\lbrace o=[\hat{o}]\in\prpq: \hspace{0,3cm} \langle \hat{o},\hat{o}\rangle_{p,q}<0 \rbrace$
\ec

\noindent a pseudo-Riemannian structure invariant under the projective action of $G$. This space is called the \textit{pseudo-Riemannian hyperbolic space of signature $(p,q-1)$}. The \textit{boundary} of $\hpq$ is the space of \textit{isotropic lines} defined by

\bc
$\partial\hpq:=\lbrace \xi=[\hat{\xi}]\in\prpq: \hspace{0,3cm} \langle\hat{\xi},\hat{\xi}\rangle_{p,q}=0 \rbrace$.
\ec

\noindent It is also equipped with the natural (transitive) action of $G$. If we denote by $P_1^{p,q}$ the (parabolic) subgroup of $G$ stabilizing an isotropic line, then

\bc
$\partial\hpq\cong G/P_1^{p,q}$.
\ec

\subsubsection{\tn{\textbf{Structure of symmetric space}}} \label{subsubsec struc sym hpq}

The action of $G$ on $\hpq$ is transitive, hence $\hpq\cong G/H^o$ where $H^o$ is the stabilizer in $G$ of the point $o\in\hpq$. For instance, when $o=[0,\dots,0,1]\in\hpq$ one has

\bc
$H^o=\left\lbrace \left[\begin{matrix}
\hat{g} & 0\\
0 & 1
\end{matrix}\right]\in G: \hspace{0,3cm} \hat{g}\in \textnormal{O}(p,q-1)  \right\rbrace$.
\ec

Fix any $o\in\hpq$. Since $o$ and $o^{\ppq}$ are transverse we can consider the matrix 

\bc
$J^o:=\tn{id}_o\oplus\left(-\tn{id}_{o^{\ppq}}\right)$. 
\ec

\noindent It follows that $H^o=\tn{Fix}(\sigma^o)$ where $\sigma^o$ is the involution of $G$ defined by
\begin{equation} \label{eq involution}
\sigma^o(g):=J^ogJ^o.
\end{equation}

\noindent Thus $\hpq\cong G/H^o$ is a symmetric space of $G$.

\begin{rem} \label{rem form on qo}
Let $o\in\hpq$ and $\lieq^o:=\lbrace d\sigma^o=-1 \rbrace$. There exists a $G$-equivariant identification

\bc
$\lieq^o\cong  T_o\hpq$
\ec

\noindent given by $X\mapsto \left. \frac{d}{dt}\right\vert_0 \exp{(tX)}\cdot o$. We denote by $\langle\cdot,\cdot\rangle$ the pull-back of the $(p,q-1)$-form on $T_o\hpq$ under this map and, for $X\in\lieq^o$, we set $\vert X\vert:=\langle X,X\rangle$\footnote{This number can be positive, negative or zero for $X\neq 0$ in $\lieq^o$.}.

Recall that $\kappa$ is the Killing form of $\mathfrak{so}(p,q)$. From explicit computations (that we omit) one can conclude that the equality
\begin{equation} \label{eq form on qo and killing}
\vert X\vert=\dfrac{1}{2(d-2)}\kappa(X,X)
\end{equation}
\noindent holds for every $X\in\lieq^o$. This justifies the choice of normalization made in Subsection \ref{subsec XG}.
\begin{flushright}
$\diamond$
\end{flushright}
\end{rem}

\begin{rem} \label{rem action of Ho en el tangente en SOpqmenos1}

Let $o\in\hpq$. Then the action of the connected component of $H^o$ containing the identity is conjugate to the action of $\tn{SO}(p,q-1)$ on $\rr^{p,q-1}$.
\begin{flushright}
$\diamond$
\end{flushright}
\end{rem}

\subsubsection{\tn{\textbf{Geodesics of $\hpq$}}}\label{subsub geod hpq}

Geodesics of $\hpq$ are the intersections of straight lines of $\prpq$ with $\hpq$. They are divided in three types:

\begin{itemize}
\item \textit{Space-like geodesics:} associated to 2-dimensional subspaces of $\rr^d$ on which $\langle\cdot, \cdot\rangle_{p,q}$ has signature $(1,1)$. They have positive speed and meet the boundary $\partial\hpq$ in two distinct points.
\item \textit{Time-like geodesics:} associated to 2-dimensional subspaces of $\rr^d$ on which $\langle\cdot, \cdot\rangle_{p,q}$ has signature $(0,2)$. They have negative speed and do not meet the boundary (they are closed). 

\item \textit{Light-like geodesics:} associated to 2-dimensional subspaces of $\rr^d$ on which $\langle\cdot, \cdot\rangle_{p,q}$ has signature $(0,1)$, that is, is degenerate but has a negative eigenvalue. They have zero speed and meet the boundary in a single point.
\end{itemize}

\noindent For a point $o\in\hpq$ we denote by $\mathscr{C}_o^{0}$ (resp. $\mathscr{C}_o^{>}$) the set of points of $\hpq$ that can be joined with $o$ by a light-like (resp. space-like) geodesic. Its closure in $\prpq$ is denoted by $\overline{\mathscr{C}_o^0}$ (resp. $\overline{\mathscr{C}_o^>}$). 

\subsubsection{\tn{\textbf{Light-cones}}}\label{subsub light cones}

The following lemma is proved by Glorieux-Monclair in \cite[Lemma 2.2]{GM}.

\begin{lema}\label{lema geod between o and xi}
Let $o\in\hpq$. Then $\overline{\mathscr{C}_o^0}\cap\partial\hpq=o^{\ppq}\cap\partial\hpq$.
\begin{flushright}
$\square$
\end{flushright}
\end{lema}

\subsubsection{\tn{\textbf{Lenghts of space-like geodesics}}}\label{subsub lengths of spacelike geod}

For a point $o'$ in $\mathscr{C}_o^>$ we denote by $\ell_{o,o'}$ the length of the geodesic segment connecting $o$ with $o'$. For instance the geodesic
\begin{equation} \label{eq geod alpha}
s\mapsto [\sinh(s),0\dots,0,\cosh(s)]\in\hpq
\end{equation}

\noindent is parametrized by arc-length.

\subsubsection{\tn{\textbf{Space-like copies of $\mathbb{H}^p$}}} \label{subsub space copies Hp}
Let $\pi$ be a $(p+1)$-dimensional subspace of $\rr^d$ of signature $(p,1)$. Then $\pp(\pi)\cap\hpq$ identifies with

\bc
$\lbrace o=[\hat{o}]\in\mathbb{P}(\rr^{p,1}) \hspace{0,3cm} \langle \hat{o},\hat{o}\rangle_{p,1}<0 \rbrace$.
\ec

\noindent It follows that $\pp(\pi)\cap\hpq$ is a totally geodesic isometric copy of $\mathbb{H}^p$ inside $\hpq$. Moreover this sub-manifold is space-like, in the sense that any of its tangent vectors has positive norm.

\subsubsection{\tn{\textbf{End points of space-like geodesics}}} \label{subsub endign of spacelike}

Let $o$ be a point in $\hpq$. Note that $J^o$ preserves the form $\langle\cdot,\cdot\rangle_{p,q}$ and thus acts on $\partial\hpq$. Set

\bc
$\mathcal{O}^o:=\lbrace \xi\in\partial\hpq:\hspace{0,3cm} J^o\cdot \xi\neq \xi\rbrace$. 
\ec

\begin{prop} \label{prop fijos de Jo en borde}
Let $o\in\hpq$. Then the following equalities hold:

\begin{equation*}
\begin{split}
\mathcal{O}^o & = \lbrace \xi\in\partial\hpq: \hspace{0,3cm} J^o\cdot \xi\notin\xi^{\ppq}\rbrace\\ & = \partial\hpq\setminus o^{\ppq} \\ & =\partial\hpq\setminus \overline{\mathscr{C}_o^0}.
\end{split}
\end{equation*}

\end{prop}

We conclude that, unless $q=1$, the set $\mathcal{O}^o$ is not the whole boundary of $\hpq$.

\begin{proof}[Proof of Proposition \ref{prop fijos de Jo en borde}]

The equality $\partial\hpq\setminus o^{\ppq}=\partial\hpq\setminus \overline{\mathscr{C}_o^0}$ is a consequence of Lemma \ref{lema geod between o and xi}. The other equalities follow from definitions.

\end{proof}

\section{Generalized Cartan decompositions} \label{sec generalized cartan}
\setcounter{equation}{0}

The goal of this section is to define two generalized Cartan projections and to provide a link between them and Theorems \ref{teorema A} and \ref{teorema B}. The first one (Subsection \ref{subsec KexpliebH}) is called the \textit{polar projection} of $G$ and it is well-known. The second one (Subsection \ref{subsec HexpliebH}) is new and can only be defined for elements in $G$ that satisfy some special property with respect to the choice of the basepoint $o$.

\subsection{Notations} \label{subsec notations generalized cartan}

Through this section we fix a point $o\in\hpq$ and let $H^o=\textnormal{Fix}(\sigma^o)$ be its stabilizer in $G$ (c.f. Subsection \ref{subsubsec struc sym hpq}). Let $\lieh^o$ be the Lie algebra of fixed points of $d\sigma^o$ and $\lieq^o:=\lbrace d\sigma^o=-1\rbrace$. One has the following decomposition of the Lie algebra $\lieg$ of $G$:

\bc

$\lieg=\lieh^o\oplus\lieq^o$.

\ec

\noindent Moreover, this decomposition is orthogonal with respect to the Killing form of $\lieg$.

Let $\tau$ be a Cartan involution  commuting with $\sigma^o$: such involutions always exist and two of them differ by conjugation by an element in $H^o$ (see Matsuki \cite[Lemma 4]{Mat}). Let $K^{\tau}:=\textnormal{Fix}(\tau)$, which is a maximal compact subgroup of $G$. Let $\liep^\tau$ and $\liek^\tau$ be the subspaces defined in Subsection \ref{subsec XG}. As $\sigma^o$ and $\tau$ commute, the following holds:

\bc
$\lieg=(\liep^{\tau}\cap\lieq^o) \oplus(\liep^{\tau}\cap\lieh^o)\oplus(\liek^{\tau}\cap\lieq^o)\oplus(\liek^{\tau}\cap\lieh^o)$.
\ec

\noindent Let $\lieb\subset\liep^{\tau}\cap\lieq^o$ be a (necessarily abelian) maximal subalgebra: two of them differ by conjugation by an element in $K^{\tau}\cap H^o$. We will consider closed Weyl chambers in $\lieb$ corresponding to positive systems of restricted roots of $\lieb$ in $\lieg^{\sigma^o\tau}:= (\liep^{\tau}\cap\lieq^o)\oplus (\liek^{\tau}\cap\lieh^o)$. These closed Weyl chambers will be denoted by $\lieb^+$.

\begin{ex} \label{ex explicit example}
Let $o=[0,\dots,0,1]$. Then $H^o$ is the upper left corner embedding of $\textnormal{O}(p,q-1)$ in $G$ and the involution $\sigma^o$ is obtained by conjugation by $J^o=\textnormal{diag}(-1,\dots,-1,1)$. One sees that $\lieh^o$ equals the upper left corner embedding of $\mathfrak{so}(p,q-1)$ in $\mathfrak{so}(p,q)$ and that

\bc

$\lieq^o=\left\lbrace\left(\begin{matrix}
0 & 0 & Y_1\\
0 & 0 & Y_2\\
Y_1^t & -Y_2^t & 0
\end{matrix}\right): \hspace{0,3cm} Y_1\in \textnormal{M}(p\times 1,\rr), \hspace{0,3cm} Y_2\in \textnormal{M}((q-1)\times 1,\rr)\right\rbrace$.
\ec

Let $\tau$ be the Cartan involution of Example \ref{ex explicit cartan involution}. One observes that $\tau$ commutes with $\sigma^o$ and

\bc

$\liep^{\tau}\cap\lieq^o=\left\lbrace X\in\lieq^o: \hspace{0,3cm} Y_2=0\right\rbrace$ \hspace{0,5cm} $\liek^{\tau}\cap\lieq^o=\left\lbrace X\in\lieq^o: \hspace{0,3cm} Y_1=0\right\rbrace$ .
\ec

\noindent Pick $\lieb$ to be the subset of $\liep^{\tau}\cap\lieq^o$ of matrices with $Y_1$ of the form

\bc
$\left(\begin{matrix}
s\\
0\\
\vdots\\
0
\end{matrix}\right)$
\ec

\noindent for some $s\in\rr$: this is a maximal subalgebra of $\liep^{\tau}\cap\lieq^o$. A closed Weyl chamber $\lieb^+$ is defined by the inequality $ s\geq 0$. 
\begin{flushright}
$\diamond$
\end{flushright}
\end{ex}

The following remark will be used repeatedly in the sequel.

\begin{rem} \label{rem J preserva norma}
Even though $G$ does not act on $\rr^d$, it makes sense to ask if an element $g$ of $G$ preserves a norm on $\rr^d$ (this notion does not depend on the choice of a lift of $g$ to $\tn{SL}(d,\rr)$). Given a Cartan involution $\tau$ commuting with $\sigma^o$, let $\Vert\cdot\Vert_\tau$ be a norm on $\rr^d$ preserved by $K^\tau$. We claim that this norm is preserved by $J^o$. Indeed, this is obvious for the choices of Example \ref{ex explicit example} and follows in general by conjugating by an element $g$ in $G$ that takes $[0,\dots,0,1]$ to the point $o$.
\begin{flushright}
$\diamond$
\end{flushright}
\end{rem}

\subsection{The sub-manifold $S^o$}\label{subsec So} Define

\bc
$S^o:=\lbrace\tau\in X_G:\hspace{0,3cm} \tau\sigma^o=\sigma^o\tau\rbrace$.
\ec

\begin{rem}\label{rem So space of hp trough o}
Recall from Remark \ref{rem Xg space of qplanes} that $X_G$ can be identified with the space of $q$-dimensional negative definite subspaces of $\rr^d$. Under this identification $S^o$ corresponds to the set of subspaces that contain the line $o$. By considering the $\langle\cdot,\cdot\rangle_{p,q}$-orthogonal complement we see that $S^o$ parametrizes the space of totally geodesic space-like copies of $\hp$ inside $\hpq$ passing through $o$ (c.f. Subsection \ref{subsub space copies Hp}).
\begin{flushright}
$\diamond$
\end{flushright}
\end{rem}

Using the fact that two elements of $S^o$ differ by conjugation by an element in $H^o$ one observes that for any $\tau\in S^o$ the following holds

\bc
$S^o=H^o\cdot\tau$.
\ec  

\noindent Further, the group $H^o$ has several connected components but one can see that the connected component containing the identity acts transitively on $S^o$. Hence $S^o$ is connected and one can show that

\bc
$S^o=\exp(\liep^{\tau}\cap\lieh^o)\cdot\tau$. 
\ec

\noindent It follows that $S^o$ is a totally geodesic sub-manifold of $X_G$ and $T_\tau S^o\cong \liep^{\tau}\cap\lieh^o$ (see \cite[Theorem 7.2 of Ch. IV]{Hel}).

\subsection{$K\exp(\lieb^+)H$-decomposition}\label{subsec KexpliebH}

For the rest of this section we fix a Cartan involution $\tau\in S^o$, a maximal subalgebra $\lieb\subset\liep^\tau\cap\lieq^o$ and a closed Weyl chamber $\lieb^+\subset\lieb$. By Schlichtkrull \cite[Proposition 7.1.3]{Sch} the following decomposition of $G $ holds:
\begin{equation}\label{eq polar decomposition}
G=K^{\tau}\exp(\lieb^+)H^o
\end{equation}
\noindent where the $\exp(\lieb^+)$-component is uniquely determined and one can define
\begin{equation} \label{eq def nu}
b^\tau:G\too\lieb^+
\end{equation}

\noindent by taking the $\textnormal{log}$ of this component. This is a continuous map called the \textit{polar projection} of $G$ associated to the choice of $\tau$ and $\lieb^+$. It generalizes the usual Cartan projection of $G$.

\begin{rem} \label{rem btau is proper}

Note that $b^\tau$ is not proper (unless $q=1$). However it descends to a map $\hpq\cong G/H^o\too\lieb^+$ which, by definition, is proper.
\begin{flushright}
$\diamond$
\end{flushright}
\end{rem}

We now discuss geometric interpretations of the polar projection $b^\tau$. The geometric interpretation in $\hpq$ follows Kassel-Kobayashi \cite[p.151]{KK}, while the geometric interpretation in $X_G$ is inspired by the work of Oh-Shah \cite{OS} for the case $p=1$ and $q=3$.

Let us begin with the interpretation in the pseudo-Riemannian setting. By Remark \ref{rem So space of hp trough o}, the choice of $\tau\in S^o$ determines a totally geodesic space-like copy of the $p$-dimensional hyperbolic space, inside $\hpq$ and passing through $o$. We denote this copy by $\hp_\tau$. From explicit computations one can show that 

\bc
$\hp_\tau=\exp(\liep^\tau\cap\lieq^o)\cdot o$.
\ec

\noindent In particular $\hp_\tau$ contains the geodesic ray $\exp(\lieb^+)\cdot o$ starting from $o$. Equality (\ref{eq polar decomposition}) tell us that for every $g$ in $G$ the point $g\cdot o$ lies in the $K^\tau$-orbit of $o_g:=\exp(b^\tau(g))\cdot o$ (see Figure \ref{fig interp geom polar}). The geometric interpretation of the polar projection is now clear: the number $\vert b^\tau(g)\vert^{\frac{1}{2}}$ equals the length of the geodesic segment connecting $o$ with $o_g$\footnote{Recall that $\vert\cdot\vert$ is the form on $\lieq^o$ defined in Remark \ref{rem form on qo}.}.

\begin{figure}[h!]
\bc
\scalebox{0.8}{%
\begin{overpic}[scale=1, width=1\textwidth, tics=5]{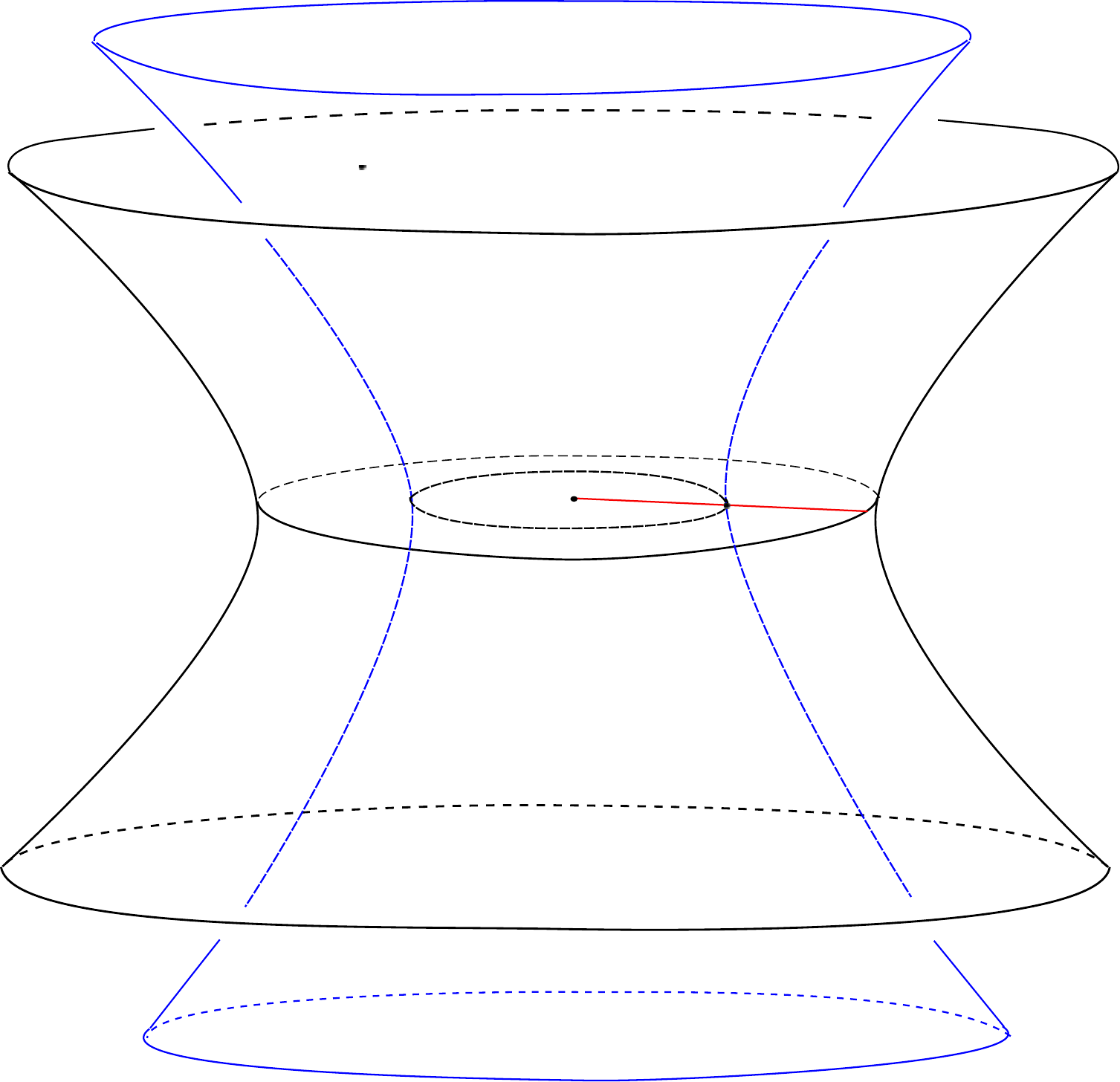}

 \put (30,79) { \Large$g\cdot o$}
  \put (87,92) { \Large\tc{blue}{$K^\tau \cdot o_g$}}
    \put (78,50) { \Large\tc{red}{$\exp(\lieb^+)\cdot o$}}
        \put (65,53) { \Large$o_g$}
        \put (48,52) { \Large$o$}
        \put (45,35) { \Large$\hpq$}
        \put (29,51) { \Large$\hp_\tau$}
        \put (5,35) { \Large$\partial\hpq$}
 \end{overpic}}
 
\hspace{0,3cm}

\begin{changemargin}{3cm}{3cm}    
    \caption{Geometric interpretation of polar projection in $\hpq$.}
  \label{fig interp geom polar}
\end{changemargin}

  \ec
\end{figure}
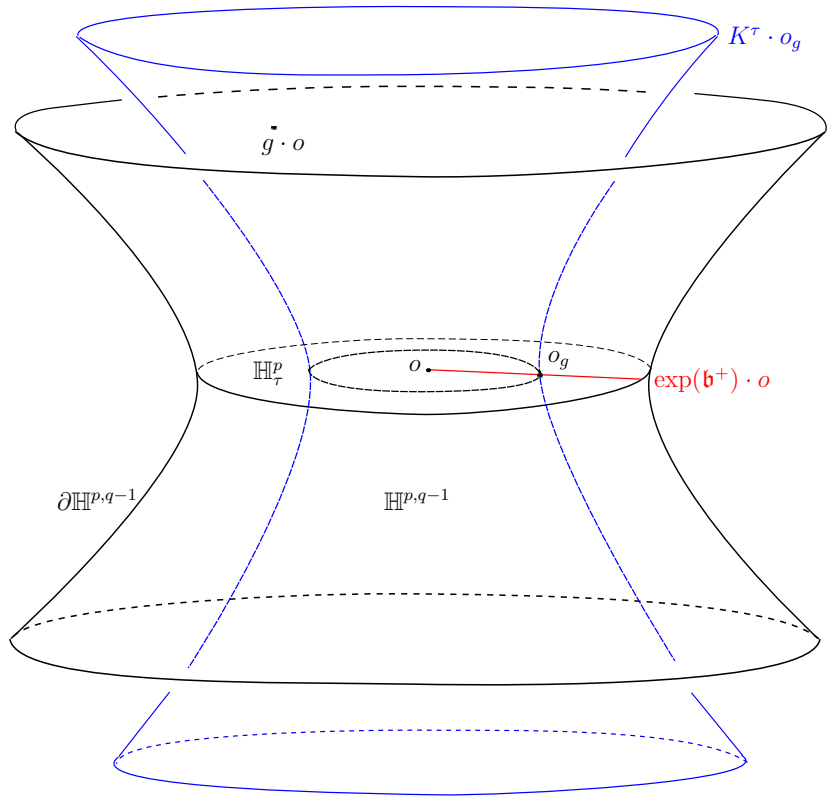

We now turn our attention to the Riemannian symmetric space $X_G$.

\begin{prop} \label{prop interpretation of btau in symg}
For every $g$ in $G$ one has 

\bc
$\vert b^\tau(g)\vert^{\frac{1}{2}}=d_{\symg}(g^{-1}\cdot \tau,S^o)$.
\ec
\end{prop}

\begin{proof}
The function $g\mapsto d_{\symg}(g^{-1}\cdot\tau,S^o)$ is $K^{\tau}$-invariant on the left and $H^o$-invariant on the right, hence it suffices to check that the equality of the statement holds when $g=\exp(X)$ for some $X\in\lieb^+$.

Since $X_G$ is non-positively curved, there exists a unique geodesic through $\exp(-X)\cdot \tau$ which is orthogonal to $S^o=\exp(\liep^\tau\cap\lieh^o)\cdot \tau$. This geodesic is $\exp(\lieb)\cdot \tau$ and intersects $S^o$ in $\tau$, hence

\bc
$d_{X_G}(\exp(-X)\cdot\tau,S^o)=d_{X_G}(\exp(-X)\cdot\tau,\tau)$.
\ec

\noindent Thanks to Remark \ref{rem form on qo} and (\ref{eq distance in XG and killing}) the proof is complete.

\end{proof}

We finish this subsection with a linear algebraic interpretation of the polar projection. Let $\Vert\cdot\Vert_{\tau}$ be a norm on $\rr^d$ invariant under the action of $K^\tau$.

\begin{prop} \label{prop computing nu}
For every $g$ in $G$ one has
\bc
$ \vert b^\tau(g)\vert^{\frac{1}{2}}=\frac{1}{2}\log\Vert J^o gJ^og^{-1}\Vert_{\tau} $.
\ec

\end{prop}

\begin{proof}

We prove the proposition for the particular choices of Example \ref{ex explicit example}, the general case follows from this one by conjugating by appropriate elements of $G$.

By Remark \ref{rem J preserva norma} the matrix $J^o$ preserves $\Vert\cdot\Vert_\tau$ thus  

\bc
$\frac{1}{2}\log\Vert J^o gJ^og^{-1}\Vert_{\tau}=\frac{1}{2}\log\Vert gJ^og^{-1}\Vert_{\tau}$.
\ec

\noindent The map $g\mapsto\frac{1}{2}\log\Vert  gJ^og^{-1}\Vert_{\tau}$ is $K^{\tau}$-invariant on the left and $H^o$-invariant on the right, hence it remains to check that the equality of the statement holds on $\exp(\lieb^+)$. Let $X\in\lieb^+$, that is,

\bc
$X=\left(\begin{matrix}
 &   & &  & s\\
  & &   & 0 & \\
 & & \iddots &  & \\
 &  0  &  &  & \\
s &  & & &   
\end{matrix}\right)$
\ec

\noindent for some $s\geq 0$. Since $X\in\lieq^o$, one has $J^o\exp(-X)=\exp(X)J^o$ and thus

\bc
$\vert X\vert^\frac{1}{2}=s=\frac{1}{2}\log\Vert  \exp(X)J^o\exp(-X)\Vert_{\tau}$.
\ec
\end{proof}

\subsection{$H\exp(\lieb^+)H$-decomposition}\label{subsec HexpliebH}

Recall from Subsection \ref{subsub geod hpq} the definition of the set $\mathscr{C}^>_o$ and define

\bc
$\mathscr{C}_{o,G}^>:=\lbrace g\in G:\hspace{0,3cm} g\cdot o\in \mathscr{C}_{o}^>\rbrace$. 
\ec

\begin{prop}\label{prop HBH}
For every $g$ in $\mathscr{C}_{o,G}^>$ one can write

\bc
$g=h\exp (X)h' $
\ec

\noindent for some $h,h'\in H^o$ and a unique $X\in\lieb^+$.
\end{prop}

It is clear that this decomposition of $g$ can only hold when $g\in\mathscr{C}_{o,G}^>$.

\begin{proof}[Proof of Proposition \ref{prop HBH}]

Take $h$ in $H^o$ such that $h^{-1}g\cdot o\in\exp(\lieb^+)\cdot o$. There exists then $X\in\lieb^+$ and $h'\in H^o$ such that $h^{-1}g=\exp(X)h'$. Note that $X$ is unique since it is determined by the length of the geodesic segment connecting $o$ with $g\cdot o$.

\end{proof}

We define the map
\begin{equation} \label{eq def de bo}
b^o:\mathscr{C}_{o,G}^>\too\lieb^+: \hspace{0,3cm} g\mapsto b^o(g)
\end{equation}

\noindent where $g=h\exp(b^o(g))h'$ for some $h,h'\in H^o$. Note that $b^o$ descends to the quotient $\mathscr{C}_{o}^>$ but this map is not proper (compare with Remark \ref{rem btau is proper}).

\begin{prop}\label{prop ell dXG y vertboverto}
For every $g$ in $\mathscr{C}_{o,G}^>$ one has

\bc
$\ell_{o,g\cdot o}=\vert b^o(g)\vert^{\frac{1}{2}}=d_{X_G}(S^o,g\cdot S^o)$.
\ec
\end{prop}

\begin{proof}

The first equality was already discussed in the proof of Proposition \ref{prop HBH}. For the second one write $g=h\exp(b^o(g))h'$. Since $S^o=H^o\cdot \tau$ we have

\bc
$d_{\symg}(S^o,h\exp(b^o(g))h'\cdot S^o)=d_{\symg}(H^o\cdot\tau,\exp(b^o(g))H^o\cdot\tau)$.
\ec

Set $X:=b^o(g)$. If $X=0$ there is nothing to prove, so assume $X\neq 0$. In that case $H^o\cdot \tau$ is disjoint from $\exp(X)H^o\cdot\tau$: since the action of $\lieb$ on the geodesic $\exp(\lieb)\cdot\tau$ is free, this follows from the fact that $X_G$ is non-positively curved and the fact that $\exp(\lieb)\cdot\tau$ intersects orthogonally $H^o\cdot\tau$ (resp. $\exp(X)H^o\cdot\tau$) in $\tau$ (resp. $\exp(X)\cdot\tau$).

\begin{cla} \label{claim in prop interp geom of bo}

Take $\tau'\in H^o\cdot\tau$ and $\tau''\in \exp(X)H^o\cdot\tau$. Then the following holds:
\bc
$d_{X_G}(\tau',\tau'')\geq d_{X_G}(\tau,\exp(X)\cdot\tau)$.
\ec
\end{cla}

\begin{proof}[Proof of Claim \ref{claim in prop interp geom of bo}]

Let $\beta_1\subset H^o\cdot \tau$ (resp. $\beta_2\subset \exp(X)H^o\cdot\tau$) be the unit-speed geodesic connecting $\beta_1(0)=\tau$ (resp. $\beta_2(0)=\exp(X)\cdot \tau$) with $\tau'$ (resp. $\tau''$). Then $\beta_1$ and $\beta_2$ are disjoint and from the fact that $X_G$ is non-positively curved follows that the map
\bc
$(t,s)\mapsto d_{X_G}(\beta_1(t),\beta_2(s))$
\ec
\noindent is smooth (see Petersen \cite[p.129]{Pet}). Moreover, since $\exp(\lieb)\cdot \tau$ is orthogonal both to $H^o\cdot\tau$ and $\exp(X)H^o\cdot\tau$ we conclude that the differential at $(0,0)$ of this map is zero. 

Take $t_0>0$ such that $\beta_1(t_0)=\tau'$ and a positive $a$ such that the geodesic $t\mapsto\beta_2(at)$ equals $\tau''$ in $t_0$. By Busemann \cite[Theorem 3.6]{Bus} the map

\bc
$t\mapsto d_{X_G}(\beta_1(t),\beta_2(at))$
\ec

\noindent is convex. Since it has a critical point at $t=0$ the proof of the claim is finished.

\end{proof}

Thanks to Remark \ref{rem form on qo} and (\ref{eq distance in XG and killing}) the proof of Proposition \ref{prop ell dXG y vertboverto} is now complete.

\end{proof}

Recall that $\lambda_1(g)$ denotes the logarithm of the spectral radius of $g\in G$.

\begin{prop}\label{prop linear alg interpr of bo}
For every $g$ in $\mathscr{C}_{o,G}^>$ one has

\bc
$\vert b^o(g)\vert^{\frac{1}{2}}=\frac{1}{2}\lambda_1(J^ogJ^og^{-1})$.
\ec

\end{prop}

\begin{proof}

It suffices to prove the proposition for the choices of $o$ and $\lieb^+$ of Example \ref{ex explicit example}. Write $g=h\exp(b^o(g))h' $ with

\bc
$b^o(g)=\left(\begin{matrix}
 &  & &   & s\\
 & & & 0 & \\
 & &  \iddots &    & \\
 & 0 & & & \\
s &  &   & &
\end{matrix}\right)$
\ec

\noindent for some $s\geq 0$. We have $\vert b^o(g)\vert^\frac{1}{2}=s$. On the other hand, $J^o$ commutes with elements of $H^o$ and thus the number $\frac{1}{2}\lambda_1(J^ogJ^og^{-1})$ equals to

\bc
$\frac{1}{2}\lambda_1(J^oh\exp(b^o(g))J^o\exp(b^o(g))^{-1}h^{-1})=\frac{1}{2}\lambda_1(J^o\exp(b^o(g))J^o\exp(b^o(g))^{-1})$.
\ec

\noindent Since $b^o(g)\in\lieq^o$ we have $J^o\exp(b^o(g))^{-1}=\exp(b^o(g))J^o$ and the proof is complete.
\end{proof}

\section{Proximality} \label{sec proximality}
\setcounter{equation}{0}

In this section we recall basic facts on product of proximal matrices, the main one being Benoist's Theorem \ref{teo benoist}. This results are well-known but we provide proofs for those which are not explicitly stated in the literature (the reader familiarized with these concepts may skip this section). Standard references are the works of Benoist \cite{Ben3,Ben4,Ben1}.

\subsection{Notations and basic definitions}\label{subsec notation definitions in sec proximality}

A norm $\Vert \cdot \Vert$ on $\rr^d$ will be fixed in the whole section. For $\xi_{1},\xi_2\in\mathbb{P}(\rr^d)$ define the distance

\bc
$d(\xi_1,\xi_2):=\inf\lbrace \Vert v_{\xi_1}-v_{\xi_2}\Vert: \hspace{0,3cm} v_{\xi_i}\in \xi_i \textnormal{ and } \Vert v_{\xi_i}\Vert =1 \textnormal{ for all } i=1,2 \rbrace$.
\ec

\noindent Let $\gr$ be the Grassmannian of $(d-1)$-dimensional subspaces of $\rr^d$. There exists a $G$-equivariant identification $\pp((\rr^d)^*)\too \gr $ given by 

\bc
$\theta\mapsto\ker\theta$
\ec

\noindent where the action of $G$ on the left side is given by $g\cdot\theta:=\theta\circ g^{-1}$. This identification will be used from now on whenever convenient.

For $\eta_1,\eta_2\in\gr$ we let

\bc
$d(\xi_1,\eta_1):=\min\lbrace d(\xi_1,\xi):\hspace{0,3cm} \xi\in\pp(\eta_1)\rbrace$
\ec

\noindent and we denote by $d^*(\eta_1,\eta_2)$ the distance on $\pp((\rr^d)^*)$ induced by the operator norm on $(\rr^d)^*$. Given a positive $\varepsilon$ we set

\bc

$b_\varepsilon(\xi_1):=\lbrace \xi\in\pp(\rr^d):\hspace{0,3cm} d(\xi_1,\xi)<\varepsilon\rbrace$
\ec

\noindent and

\bc

$B_\varepsilon(\eta_1):=\lbrace \xi\in\pp(\rr^d):\hspace{0,3cm} d(\xi,\eta_1)\geq\varepsilon\rbrace$.
\ec

On the other hand, let

\bc
$\ppc:=\lbrace (\theta,v)\in\pp((\rr^d)^*)\times\pp(\rr^d): \hspace{0,3cm} v\notin\ker\theta \rbrace$
\ec

\noindent and

\bc
$\ppcu:=\lbrace (\theta,v,\phi,u)\in\ppc\times\ppc: \hspace{0,3cm} v\notin\ker\phi \tn{ and } u\notin\ker\theta \rbrace$.
\ec

\noindent Observe that 
\begin{equation} \label{eq defi gcursiva}
\mathscr{G}_{\Vert\cdot\Vert}=\mathscr{G}:\ppc\too\rr: \hspace{0,3cm} \mathscr{G}( \theta,v):=\log\dfrac{\left\vert\theta(v)\right\vert}{\Vert \theta\Vert\Vert v\Vert}	
 \end{equation} 
\noindent is well-defined. Similarly the following map is well-defined
\begin{equation}\label{eqdef crossratio}
\bb:\ppcu\too\rr: \hspace{0,3cm} \bb(\theta,v,\phi,u):= \log\left\vert\dfrac{\theta(u)}{\theta(v)}	\dfrac{\phi(v)}{ \phi(u)}\right\vert	
 \end{equation}

\noindent and is called de \textit{cross-ratio} of $(\theta,v,\phi,u)$\footnote{Sometimes $e^{\bb}$ is called the cross-ratio.}. Both $\mathscr{G}$ and $\bb$ are continuous.

\subsection{Product of proximal matrices}\label{subsec product of proximal}

Given $g$ in $\textnormal{End}(\rr^d)\setminus\lbrace 0\rbrace$ we denote by

\bc
$\lambda_1(g)\geq\dots\geq\lambda_d(g)$
\ec

\noindent the logarithms of the moduli of the eigenvalues of $g$, repeated with multiplicity (we use the convention $\log 0=-\infty$). The matrix $g$ is said to be \textit{proximal} in $\mathbb{P}(\rr^d)$ if $\lambda_1(g)$ is simple. In that case we let $g_+$ (resp. $g_-$) to be the attractive fixed line (resp. repelling fixed hyperplane) of $g$ in $\mathbb{P}(\rr^d)$. Note that if $g$ is non invertible then $g_-$ contains the kernel of $g$.

We now define a quantified version of proximality. The definition that we propose is (slightly) weaker than the one given by Benoist in \cite{Ben3,Ben4,Ben1}. We provide proofs of the basic facts established in those works when necessary.

\begin{dfn} \label{dfn of repsilon proximality}
Let $0<\varepsilon \leq r$ and $g\in\textnormal{End}(\rr^d)\setminus\lbrace 0\rbrace$ be a proximal matrix. The matrix $g$ is called \textit{$(r,\varepsilon)$-proximal} if $d(g_+,g_-)\geq 2r$ and $g\cdot B_\varepsilon(g_-)\subset b_\varepsilon(g_+)$. 
\end{dfn}

\begin{lema}[Benoist {\cite[Corollaire 6.3]{Ben3}}]\label{lema benoist proximal comp vasing y vap}
Let $0<\varepsilon\leq r$. There exists a constant $ c_{r,\varepsilon}>0$ such that for every $(r,\varepsilon)$-proximal matrix $g$ one has

\bc
$\log\Vert g \Vert - c_{r,\varepsilon}\leq \lambda_1(g)\leq \log\Vert g \Vert$.
\ec
\begin{flushright}
$\square$
\end{flushright}
\end{lema}

The following criterion of $(r,\varepsilon)$-proximality will be very useful in the sequel.

\begin{lema}[Benoist {\cite[Lemme 6.2]{Ben3}}]\label{lema benoist lemma 1.2}
Let $g$ be an element in $\tn{End}(\rr^d)\setminus\lbrace 0\rbrace$, $\eta\in\gr$, $\xi\in\mathbb{P}(\rr^d)$ and $0<\varepsilon\leq r$. If $d(\xi,\eta)\geq 6r$ and $g\cdot B_\varepsilon(\eta)\subset b_\varepsilon(\xi)$ then $g$ is $(2r,2\varepsilon)$-proximal with $d(g_+,\xi)\leq\varepsilon$ and $d^*(g_-,\eta)\leq\varepsilon$.

\end{lema}

\begin{proof}
Consider the Hilbert distance on the convex set $B_\varepsilon(\eta)$ (see \cite{Ben2}). The condition $g\cdot B_\varepsilon(\eta)\subset b_\varepsilon(\xi)$ implies that $g$ is contracting for this metric and thus has a unique fixed point in $B_\varepsilon(\eta)$, which belongs in fact to $b_\varepsilon(\xi)$. The proof now finishes as in \cite[Lemme 6.2]{Ben3}.

\end{proof}

\begin{cor}[Benoist {\cite[Lemme 1.4]{Ben1}}]\label{cor product of proximal transverse is proximal}
Let $0<\varepsilon\leq r$. If $g_1$ and $g_2$ are $(r,\varepsilon)$-proximal and satisfy 

\bc
$d(g_{1_+},g_{2_-})\geq 6r$ and $d(g_{2_+},g_{1_-})\geq 6r$
\ec

\noindent then $g_1g_2$ is $(2r,2\varepsilon)$-proximal. 
\begin{flushright}
$\square$
\end{flushright}
\end{cor}

Let $g_1$ and $g_2$ be two matrices as in Corollary \ref{cor product of proximal transverse is proximal}. The goal now is to state a theorem (Theorem \ref{teo benoist}) which provides a comparison between the spectral radius and operator norm of $g_1g_2$ in terms of the spectral radii of $g_1$ and $g_2$ and the maps $\mathscr{G}$ and $\bb$.

\begin{lema}\label{lemma atracrep of product}

Fix $r>0$ and $\delta> 0$. For every $\varepsilon$ small enough, the following property is satisfied: for every pair of $(r,\varepsilon)$-proximal elements $g_1$ and $g_2$ such that

\bc
$d(g_{1_+},g_{2_-})\geq 6r$ and $d(g_{2_+},g_{1_-})\geq 6r$
\ec

\noindent one has

\bc
$\vert \mathscr{G}(g_{2_-},g_{1_+})-\mathscr{G}((g_1g_2)_-,(g_1g_2)_+)\vert<\delta$.
\ec

\end{lema}

\begin{proof}
For every $0<\varepsilon\leq r$, consider the compact set $C_{r,\varepsilon}$ of pairs $(g_1,g_2)$ of norm-one $(r,\varepsilon)$-proximal matrices in $\tn{End}(\rr^d)\setminus\lbrace 0\rbrace$ satisfying 

\bc
$d(g_{1_+},g_{2_-})\geq 6r$ and $d(g_{2_+},g_{1_-})\geq 6r$.
\ec

\noindent The function
\bc
$(g_1,g_2)\mapsto \vert \mathscr{G}(g_{2_-},g_{1_+})-\mathscr{G}((g_1g_2)_-,(g_1g_2)_+)\vert$
\ec

\noindent is continuous and equals zero on $C_r:=\displaystyle\cap_{\varepsilon >0}C_{r,\varepsilon}\subset \textnormal{End}(\rr^d)\setminus\lbrace 0\rbrace$.

\end{proof}

\begin{teo}[Benoist {\cite[Lemme 1.4]{Ben1}}] \label{teo benoist}

Fix $r>0$ and $\delta> 0$. Then for every $\varepsilon$ small enough, the following properties are satisfied: for every pair of $(r,\varepsilon)$-proximal elements $g_1$ and $g_2$ such that

\bc
$d(g_{1_+},g_{2_-})\geq 6r$ and $d(g_{2_+},g_{1_-})\geq 6r$
\ec

\noindent one has:

\begin{enumerate}
\item The number

\bc

$\left\vert \lambda_1( g_1g_2)  - (\lambda_1(g_1)+\lambda_1(g_2))-\bb(g_{1_-},g_{1_+},g_{2_-},g_{2_+}) \right\vert$

\ec

\noindent is less than $\delta$.

\item The number

\bc
$\left\vert \log\Vert g_1g_2 \Vert - (\lambda_1(g_1)+\lambda_1(g_2))-\bb(g_{1_-},g_{1_+},g_{2_-},g_{2_+})+\mathscr{G}(g_{2_-},g_{1_+}) \right\vert$
\ec

\noindent is less than $\delta$.

\end{enumerate}

\end{teo}

\begin{proof}

\begin{enumerate}
\item[(1)] See \cite[Lemme 1.4]{Ben1}.

\item[(2)] Let $\varepsilon$ be as in (1). For every $g_1$ and $g_2$ as in the statement, Corollary \ref{cor product of proximal transverse is proximal} implies that $g_1g_2$ is $(2r,2\varepsilon)$-proximal. By \cite[Lemma 5.6]{Sam} (and taking $\varepsilon$ smaller if necessary) we have

\bc
$\left\vert \log\Vert g_1g_2 \Vert - \lambda_1(g_1g_2)+\mathscr{G}((g_1g_2)_-,(g_1g_2)_+) \right\vert<\delta$.
\ec

\noindent Lemma \ref{lemma atracrep of product} finishes the proof.

\end{enumerate}

\end{proof}

\section{Projective Anosov representations} \label{sec anosov}
\setcounter{equation}{0}

Anosov representations were introduced by Labourie \cite{Lab} for surface groups and extended by Guichard-Wienhard \cite{GW} to word hyperbolic groups. In this section we recall the definition of (projective) Anosov representations and some well-known facts concerning $(r,\varepsilon)$-proximality of matrices in the image of such a representation.

\subsection{Singular values}\label{subsec cartan dec in sec anosov}

The most useful characterization of Anosov representations for our purposes is the one given in terms of \textit{singular values}. We begin by recalling this notion and we fix also some notations that we will use in the rest of the paper.

Let $\tau$ be a $q$-dimensional subspace of $\rr^d$ which is negative definite for $\langle\cdot,\cdot\rangle_{p,q}$. Consider $\langle \cdot,\cdot\rangle_\tau$ to be the inner product of $\rr^d$ that coincides with $-\langle\cdot,\cdot\rangle_{p,q}$ (resp. $\langle\cdot,\cdot\rangle_{p,q}$) on $\tau$ (resp. $\tau^{\ppq}$) and for which $\tau$ and $\tau^{\ppq}$ are orthogonal. Given $g$ in $\tn{PSL}(d,\rr)$, we let $g^{*_\tau}$ to be the adjoint operator with respect to $\langle \cdot,\cdot\rangle_\tau$. Set

\bc

$a_1^\tau(g)\geq \dots \geq a_d^\tau(g)$
\ec

\noindent to be the logarithms of the eigenvalues of $\sqrt{g^{*_\tau}g}$ repeated with multiplicity. These are called the $\tau$-\textit{singular values} of $g$. Geometrically, they represent the (logarithms of the) lengths of the semi axes of the ellipsoid which is the image by $g$ of the unit sphere

\bc
$\mathbb{S}^{d-1}_\tau:=\lbrace x\in\rr^d:\hspace{0,3cm} \langle x,x\rangle_\tau=1 \rbrace$.
\ec

Let $i=1,\dots,d-1$. Given an element $g$ in $\tn{PSL}(d,\rr)$ such that $a^\tau_i(g)>a^\tau_{i+1}(g)$ we denote by $U_i(g)$ the $i$-dimensional subspace of $\rr^d$ spanned by the $i$ biggest axes of $g\cdot \mathbb{S}^{d-1}_\tau$. We also set

\bc
$S_{d-i}(g):=U_{d-i}(g^{-1})$.
\ec

\begin{rem} \label{rem complemento de Sdmenosuno va en Uuno}
Let $\varepsilon>0$. It follows from Singular Value Decomposition (see Horn-Johnson \cite[Section 7.3 of Chapter 7]{HJ}), that there exists $L>0$ such that for every $g$ in $\tn{PSL}(d,\rr)$ satisfying $a^\tau_1(g)-a^\tau_2(g)>L$ one has

\bc
$g\cdot B_\varepsilon(S_{d-1}(g))\subset b_\varepsilon( U_1(g))$,
\ec

\noindent where $B_\varepsilon(S_{d-1}(g))$ and $b_\varepsilon( U_1(g))$ are defined as in Subsection \ref{subsec notation definitions in sec proximality}.
\begin{flushright}
$\diamond$
\end{flushright}
\end{rem}

\subsection{The definition of projective Anosov representations}\label{subsec deifnition anosov}

A lot of work has been done in order to simplify the original definition of Anosov representations, here we follow mainly the work of Bochi-Potrie-Sambarino \cite{BPS} (see also Guichard-Guéritaud-Kassel-Wienhard \cite{GGKW} or Kapovich-Leeb-Porti \cite{KLP1}).

Fix $\tau$ as in the previous subsection and let $\Gamma$ be a finitely generated group. Consider a finite symmetric generating set $S$ of $\Gamma$ and take $\vert\cdot\vert_\Gamma$ to be the associated word length: for $\gamma$ in $\Gamma$, it is the minimum number required to write $\gamma$ as a product of elements of $S$\footnote{This number depends on the choice of $S$. However, the set $S$ will be fixed from now on hence we do not emphasize the dependence on this choice in the notation.}. Let $\rho:\Gamma\too \textnormal{PSL}{(d,\rr)}$ be a representation. We say that $\rho$ is \textit{projective Anosov} if there exist positive constants $C$ and $\alpha$ such that for all $\gamma\in\Gamma$ one has
\begin{equation}\label{eq def anosov}
a^\tau_1(\rho(\gamma))-a^\tau_2(\rho(\gamma))\geq \alpha\vert\gamma\vert_\Gamma-C.
\end{equation}
By Kapovich-Leeb-Porti \cite[Theorem 1.4]{KLP3} (see also \cite[Section 3]{BPS}), condition (\ref{eq def anosov}) implies that $\Gamma$ is word hyperbolic\footnote{We refer the reader to the book of Ghys-de la Harpe \cite{GdlH} for definitions and standard facts on word hyperbolic groups.}. We assume in this paper that $\Gamma$ is non elementary. Let $\bg$ be the Gromov boundary of $\Gamma$ and $\gh$ be the set of infinite order elements in $\Gamma$. Every $\gamma$ in $\gh$ has exactly two fixed points in $\bg$: the attractive one denoted by $\gamma_+$ and the repelling one denoted by $\gamma_-$. The dynamics of $\gamma$ on $\bg$ is of type \textit{north-south}.

Fix $\rho:\Gamma\too\tn{PSL}(d,\rr)$ a projective Anosov representation. By \cite{BPS,GGKW,KLP1} we know that there exist continuous equivariant maps

\bc
$\xi:\bg\too\pp(\rr^d)$ and $\eta:\bg\too\gr$
\ec

\noindent which are \textit{transverse}, i.e. for every $x\neq y$ in $\bg$ one has
\begin{equation} \label{eq transv condition}
\xi(x)\oplus\eta(y)=\rr^d.
\end{equation}

\noindent One can see that condition (\ref{eq def anosov}) implies that for every $\gamma$ in $\gh$ the matrix $\rho(\gamma)$ is proximal. Equivariance of $\xi$ and $\eta$ implies that

\bc
$\xi(\gamma_+)=\rho(\gamma)_+$ and $\eta(\gamma_+)=\rho(\gamma^{-1})_-$.
\ec

\noindent It follows that both $\xi$ and $\eta$ are homeomorphisms onto their images. In fact, these homeo\-mor\-phisms are Hölder (see Bridgeman-Canary-Labourie-Sambarino \cite[Lemma 2.5]{BCLS}).

We denote by $\Lambda_{\rho(\Gamma)}\subset\pp(\rr^d)$ the image of $\xi$, which is called the \textit{limit set} of $\rho(\Gamma)$: it is the closure of the set of attractive fixed points in $\pp(\rr^d)$ of proximal elements in $\rho(\Gamma)$. The image of $\eta$  is called the \textit{dual limit set} of $\rho(\Gamma)$.

Here is another characterization of the limit sets which is very useful. An explicit reference is \cite[Theorem 5.3]{GGKW} (it can also be deduced from {\cite[Subsection 3.4]{BPS}}). Let $d=d_\tau$ (resp. $d^*=d^*_\tau$) be the distance on $\pp(\rr^d)$ (resp. $\pp((\rr^d)^*)$) associated to $\langle\cdot,\cdot\rangle_\tau$.

\begin{prop}\label{prop limit with S and U and Uuno cerca gammamas}
Let $\rho:\Gamma\too \tn{PSL}(d,\rr)$ be a projective Anosov representation. Then $\xi(\bg)$ (resp. $\eta(\bg)$) equals the set of accumulation points of sequences $\lbrace U_1(\rho(\gamma_n))\rbrace_n$ (resp. $\lbrace S_{d-1}(\rho(\gamma_n))\rbrace_n$) where $\gamma_n\too\infty$. Moreover, given a positive $\varepsilon$ there exists $L>0$ such that for every $\gamma$ in $\gh$ with $\vert\gamma\vert_\Gamma>L$ one has

\bc

$d(U_1(\rho(\gamma)),\rho(\gamma)_+)<\varepsilon$ and $d^*(S_{d-1}(\rho(\gamma)),\rho(\gamma)_-)<\varepsilon$.

\ec
\begin{flushright}
$\square$
\end{flushright}
\end{prop}

We are interested in projective Anosov representations whose image is contained in $G=\tn{PSO}(p,q)$. The following remark is then important for our purposes.

\begin{rem}
Let $\rho:\Gamma\too\tn{PSL}(d,\rr)$ be a projective Anosov representation. If $\rho(\Gamma)$ is contained in $G$ we say that $\rho$ is \textit{$P_1^{p,q}$-Anosov} (recall that $P_1^{p,q}$ denotes the (parabolic) subgroup of $G$ stabilizing an isotropic line). In this case, the image of $\xi$ is contained in $\partial\hpq$ and the dual map $\eta$ equals $\xi^{\ppq}$.
\begin{flushright}
$\diamond$
\end{flushright}
\end{rem}

\subsection{Proximality properties}\label{subsec proximality anosov}
The following lemma will be useful in the next section.

\begin{lema}[c.f. {\cite[Lemma 5.7]{Sam}}]\label{lema sambarino lemma 5.7}
Let $\rho:\Gamma\too \tn{PSL}(d,\rr)$ be a  projective Anosov representation and $0<\varepsilon\leq r$. Then

\bc
$\#\lbrace\gamma\in\gh: \hspace{0,3cm} d(\rho(\gamma)_+,\rho(\gamma)_-)\geq 2r \textnormal{ and } \rho(\gamma) \textnormal{ is not } (r,\varepsilon)\textnormal{-proximal}\rbrace<\infty$.
\ec

\end{lema}

\begin{proof}

Consider a sequence $\gamma_n\too\infty$ in $\gh$ such that $d(\rho(\gamma_n)_+,\rho(\gamma_n)_-)\geq 2r$ for all $n$. By Proposition \ref{prop limit with S and U and Uuno cerca gammamas} for every $n$ big enough the following holds

\bc
$b_{\frac{\varepsilon}{2}}(U_1(\rho(\gamma_n)))\subset b_{\varepsilon}(\rho(\gamma_n)_+)$
\ec

\noindent and

\bc
$B_{\varepsilon}(\rho(\gamma_n)_-)\subset B_{\frac{\varepsilon}{2}}(S_{d-1}(\rho(\gamma_n)))$.
\ec

\noindent By Remark \ref{rem complemento de Sdmenosuno va en Uuno} and (\ref{eq def anosov}) the condition $\rho(\gamma_n)\cdot B_{\varepsilon}(\rho(\gamma_n)_-)\subset b_{\varepsilon}(\rho(\gamma_n)_+)$ is sa\-tis\-fied for sufficiently large $n$.

\end{proof}

\section{The set $\pmb{\Omega}_\rho$}\label{sec the set omegarho}
\setcounter{equation}{0}

Let $\rho:\Gamma\too G$ be a $P_1^{p,q}$-Anosov representation and define

\bc
$\pmb{\Omega}_{\rho}:=\lbrace o\in \hpq:\hspace{0,3cm} J^o\cdot \xi(x)\notin\eta(x) \tn{ for all } x\in\bg\rbrace$.
\ec

\noindent This section is structured as follows. In Subsection \ref{subsec dynam on omegarho} we prove that the action of $\Gamma$ on $\pmb{\Omega}_\rho$ is properly discontinuous. Moreover, we show that if $o$ is a point in $\pmb{\Omega}_\rho$ then the geodesic connecting $o$ with $\rho(\gamma)\cdot o$ is space-like (apart from possibly finitely many exceptions $\gamma\in\Gamma$). In Subsection \ref{subsec proxim of jrhojrho} we study the matrices $J^o\rho(\gamma)J^o\rho(\gamma^{-1})$ for a point $o$ in $\pmb{\Omega}_\rho$: we apply to them Benoist's work on proximality. Finiteness of our counting functions is proved in Subsection \ref{subsec orbital counting functions}. Finally, in Subsection \ref{subsec weak triangle} we prove a proposition that will be needed in the proof of Proposition \ref{prop distribution on bg for length}.

Before we start, let us discuss some examples for which $\pmb{\Omega}_\rho$ is non empty. From Proposition \ref{prop fijos de Jo en borde} we know that the following alternative description of $\pmb{\Omega}_\rho$ holds

\bc
$\pmb{\Omega}_{\rho}=\lbrace o=[\hat{o}]\in \hpq:\hspace{0,3cm} \langle \hat{o},\hat{\xi}\rangle_{p,q}\neq 0 \tn{ for all } \xi=[\hat{\xi}]\in\Lambda_{\rho(\Gamma)}\rbrace$.
\ec

\noindent We have the following important example.

\begin{ex}\label{ex omegarho no vacio}

\item 

\begin{itemize}
\item Let $\Gamma$ be the fundamental group of a convex co-compact hyperbolic manifold of dimension $m\geq 2$ and $\iota_0:\Gamma\too \tn{SO}(m,1)$ be the holonomy representation. Fix $p\geq m$ and $q\geq 2$. Consider the embedding $\rr^{m,1}\hookrightarrow \rr^{p,q}$ given by

\bc
$\rr^{m,1}\cong \tn{span}\lbrace e_{p-m+1},\dots,e_{p+1}\rbrace$,
\ec

\noindent where $e_i$ is the vector of $\rr^d$ with all entries equal to zero except for the $i$-th entry which is equal to one. This induces a projection $j:\tn{SO}(m,1)\too G$ and a representation $\rho_0:\Gamma\too G$ defined by

\bc
$\rho_0:=j\circ\iota_0$.
\ec

\noindent Thus $\rho_0$ is $P_1^{p,q}$-Anosov, because $\iota_0$ is $P_1^{m,1}$-Anosov. The set $\pmb{\Omega}_{\rho_0}$ is non empty: every point $o\in\hpq$ for which the subspace

\bc
$\tn{span}\lbrace o,e_{p+2},\dots,e_d\rbrace$
\ec

\noindent has signature $(0,q)$ belongs to $\pmb{\Omega}_{\rho_0}$. Since the condition of being Anosov is open in the space of representations of $\Gamma$ into $G$ and the limit map $\xi$ varies continuously with the representation (see Guichard-Wienhard \cite[Theorem 5.13]{GW}), we obtain that if $\rho$ is a small deformation of $\rho_0$ then $\pmb{\Omega}_\rho$ is non empty.

\item The previous example generalizes to a large class of representations introduced by Danciger-Guéritaud-Kassel in \cite{DGK1,DGK2} called \textit{$\hpq$-convex co-compact}\footnote{These are inclusion representations induced by taking an infinite discrete subgroup $\Gamma< G$ which preserves some properly convex non empty open set $\Omega\subset\pp(\rr^d)$ whose boundary is strictly convex and of class $C^1$. One requires that $\Gamma$ preserves some \textit{distinguished} non empty convex subset of $\Omega$ on which the action is co-compact (see \cite{DGK1,DGK2} for precisions).}. Let $\Gamma <G$ be a $\hpq$-convex co-compact group and $\rho:\Gamma\too G$ be the inclusion representation, which is $P_1^{p,q}$-Anosov as proved in \cite[Theorem 1.25]{DGK2}. Let $\Omega$ be a non empty $\Gamma$-invariant properly convex open subset of $\hpq$. By \cite[Proposition 4.5]{DGK2}, $\Omega$ is contained in $\pmb{\Omega}_\rho$.

\item There exist examples of $P_1^{p,q}$-Anosov representations $\rho$ whose image is not $\hpq$-convex co-compact but satisfy $\pmb{\Omega}_\rho\neq\emptyset$ (see \cite[Examples 5.2 \& 5.3]{DGK1}).

\end{itemize}

\begin{flushright}
$\diamond$
\end{flushright}
\end{ex}

\subsection{Dynamics on $\pmb{\Omega}_\rho$}\label{subsec dynam on omegarho}

Observe that $\pmb{\Omega}_\rho$ is $\Gamma$-invariant. The following proposition is well-known, we include a proof for completeness.

\begin{prop}\label{prop action on omegarho is prop discont and limit set is limit set}
Let $\rho:\Gamma\too G$ be a $P_1^{p,q}$-Anosov representation. Then the action of $\Gamma$ on $\pmb{\Omega}_\rho$ is properly discontinuous, that is, for every compact set $C\subset\pmb{\Omega}_\rho$ one has

\bc
$\#\left\lbrace \gamma\in\Gamma: \hspace{0,3cm} \rho(\gamma)\cdot C\cap C\neq \emptyset\right\rbrace<\infty$.
\ec

\noindent Moreover, for any point $o$ in $\pmb{\Omega}_\rho$ the set of accumulation points of $\rho(\Gamma)\cdot o$ in $\hpq\cup\partial\hpq$ coincides with the limit set $\Lambda_{\rho(\Gamma)}$.
\end{prop}

\begin{proof}

Let $C\subset\pmb{\Omega}_\rho$ be a compact set and fix a norm on $\rr^d$. By definition of $\pmb{\Omega}_\rho$ we can take a positive $\varepsilon$ such that

\bc
$C\cap\displaystyle\bigcup_{x\in\bg}b_\varepsilon(\xi(x))=\emptyset$ and $C\subset\displaystyle\bigcap_{x\in\bg}B_\varepsilon(\eta(x))$.
\ec

By Proposition \ref{prop limit with S and U and Uuno cerca gammamas}, Remark \ref{rem complemento de Sdmenosuno va en Uuno} and (\ref{eq def anosov}) we know that, apart from possibly finitely many exceptions $\gamma$ in $\Gamma$, the following holds:

\bc
$b_{\frac{\varepsilon}{2}}(U_1(\rho(\gamma)))\subset \displaystyle\bigcup_{x\in\bg}b_\varepsilon(\xi(x))$,

\vspace{0,3cm}

$\displaystyle\bigcap_{x\in\bg}B_\varepsilon(\eta(x))\subset B_{\frac{\varepsilon}{2}}(S_{d-1}(\rho(\gamma)))$

\ec

\noindent and

\bc

$\rho(\gamma)\cdot B_{\frac{\varepsilon}{2}}(S_{d-1}(\rho(\gamma))) \subset b_{\frac{\varepsilon}{2}}(U_1(\rho(\gamma)))$.

\ec

\noindent For these $\gamma$ we have then that $\rho(\gamma)\cdot C$ is contained in the $\varepsilon$-neighbourhood of $\Lambda_{\rho(\Gamma)}$ and thus is disjoint from $C$.

We have shown that the action of $\Gamma$ on $\pmb{\Omega}_\rho$ is properly discontinuous and that for any point $o$ in $\pmb{\Omega}_\rho$ the accumulation points of $\rho(\Gamma)\cdot o$ belong to $\Lambda_{\rho(\Gamma)}$. Conversely, the $\Gamma$-orbit of any point in $\Lambda_{\rho(\Gamma)}$ is dense in the limit set and now the proof is complete.

\end{proof}

Let $o\in\pmb{\Omega}_\rho$ and recall the notations introduced in Subsection \ref{subsub geod hpq}. Given an open set $W\subset\partial\hpq$ disjoint from $\overline{\mathscr{C}_o^0}\cap\partial\hpq$ we denote by $\mathscr{C}_{o}^{>_W}$ the subset of $\mathscr{C}_o^{>}$ consisting of points $o'$ such that the (space-like) geodesic ray connecting $o$ with $o'$ has its end point in $W$.

The following corollary has been proved by Glorieux-Monclair \cite{GM} for $\hpq$-convex co-compact groups.

\begin{cor}\label{cor gammao in cowmayor}
Let $\rho:\Gamma\too G$ be a $P_1^{p,q}$-Anosov representation, a point $o\in\pmb{\Omega}_{\rho}$ and $W\subset\partial\hpq$ an open set containing $\Lambda_{\rho(\Gamma)}$ with closure disjoint from $\overline{\mathscr{C}_o^0}\cap\partial\hpq$. Then apart from possibly finitely many exceptions $\gamma$ in $\Gamma$ one has $\rho(\gamma)\cdot o \in\mathscr{C}_{o}^{>_W}$. In particular the geodesic joining $o$ with $\rho(\gamma)\cdot o$ is space-like.
\end{cor}

\begin{proof}

Let $C$ be the closure of $\hpq\setminus\mathscr{C}_{o}^{>_W}$ in $\hpq\cup\partial\hpq$. Note that $C$ is compact and by Proposition \ref{prop action on omegarho is prop discont and limit set is limit set} does not contain accumulation points of $\rho(\Gamma)\cdot o$, hence $\rho(\Gamma)\cdot o\cap C$ is finite.  Since $\gamma\mapsto\rho(\gamma)\cdot o$ is proper the proof is complete.

\end{proof}

\subsection{Proximality of $J^o\rho(\gamma)J^o\rho(\gamma^{-1})$} \label{subsec proxim of jrhojrho}

For the rest of the section we fix a $P_1^{p,q}$-Anosov representation $\rho:\Gamma\too G$, a point $o\in\pmb{\Omega}_\rho$ and a Cartan involution $\tau\in S^o$.

The next lemma is a direct consequence of Proposition \ref{prop limit with S and U and Uuno cerca gammamas}, transversality condition (\ref{eq transv condition}) and the definition of $\pmb{\Omega}_\rho$.

\begin{lema}\label{lema JUuno lejos de Sdmenos1}
Let $d_\tau$ be the distance on $\pp(\rr^d)$ induced by the norm $\Vert\cdot\Vert_\tau$. There exists a positive constant $D$ such that

\bc
$\#\lbrace \gamma\in\Gamma:\hspace{0,3cm} d_\tau(J^o\cdot U_1(\rho(\gamma)),S_{d-1}(\rho(\gamma^{-1})))< D\rbrace <\infty$.
\ec
\begin{flushright}
$\square$
\end{flushright}
\end{lema}

\begin{lema} \label{lema jrhojrho prox uno}
There exist $0<\varepsilon\leq r$ such that, apart from possibly finitely many exceptions $\gamma\in\Gamma$, the matrix $J^o\rho(\gamma)J^o\rho(\gamma^{-1})$ is $(r,\varepsilon)$-proximal.
\end{lema}

\begin{proof}

We apply a \textit{ping-pong} argument together with Lemma \ref{lema benoist lemma 1.2}. By Lemma \ref{lema JUuno lejos de Sdmenos1} we can take a positive constant $r$ and a finite subset $F\subset\Gamma$ such that for every $\gamma\in\Gamma\setminus F$ one has
\begin{equation}\label{eq jrhojrho prox uno}
d_\tau (J^o\cdot U_1(\rho(\gamma)),S_{d-1}(\rho(\gamma^{-1})))\geq 6r.
\end{equation}

\noindent Take $0<\varepsilon\leq r$ such that for every $\gamma\in\Gamma\setminus F$ one has

\bc
$b_\varepsilon(J^o\cdot U_1(\rho(\gamma)))\subset B_\varepsilon(S_{d-1}(\rho(\gamma^{-1})))$.
\ec

\noindent By Remark \ref{rem J preserva norma} the matrix $J^o$ preserves $d_\tau$ thus

\bc
$J^o\cdot b_\varepsilon( U_1(\rho(\gamma)))\subset B_\varepsilon( S_{d-1}(\rho(\gamma^{-1})))$.
\ec

\noindent By taking $F$ larger if necessary we have that

\bc
$\rho(\gamma^{-1})\cdot B_\varepsilon(S_{d-1}(\rho(\gamma^{-1})))\subset b_\varepsilon(U_1(\rho(\gamma^{-1})))$
\ec

\noindent holds for every $\gamma$ in $\Gamma\setminus F$. It follows that 

\bc$J^o\rho(\gamma^{-1})\cdot B_\varepsilon(S_{d-1}(\rho(\gamma^{-1})))\subset B_\varepsilon(S_{d-1}(\rho(\gamma)))$
\ec

\noindent and applying $\rho(\gamma)$ we obtain

\bc
$\rho(\gamma)J^o\rho(\gamma^{-1})\cdot B_\varepsilon(S_{d-1}(\rho(\gamma^{-1})))\subset b_\varepsilon(U_1(\rho(\gamma)))$.
\ec

\noindent Then
\bc
$J^o\rho(\gamma)J^o\rho(\gamma^{-1})\cdot B_\varepsilon(S_{d-1}(\rho(\gamma^{-1})))\subset b_\varepsilon(J^o\cdot  U_1(\rho(\gamma)))$.
\ec

\noindent By (\ref{eq jrhojrho prox uno}) and Lemma \ref{lema benoist lemma 1.2} the proof is finished.

\end{proof}

The following is a strengthening of Lemma \ref{lema jrhojrho prox uno}. It provides a link between the generalized Cartan projections $b^o$ and $b^\tau$ and the spectral radii of proximal elements in $\rho(\Gamma)$. For the remainder of the section we fix a maximal subalgebra $\lieb\subset\liep^\tau\cap\lieq^o$ and a closed Weyl chamber $\lieb^+$.

\begin{lema}\label{lema jrhojrho prox dos}

Fix any $\delta>0$ and  $A$ and $B$ two compact disjoint sets in $\bg$. Then there exist $0< \varepsilon\leq r$ such that, apart from possibly finitely many exceptions $\gamma\in \gh$ with $\gamma_-\in A$ and $\gamma_+\in B$, the following holds:

\begin{enumerate}
\item The matrices $J^o\rho(\gamma)J^o$ and $\rho(\gamma^{-1})$ are $(r,\varepsilon)$-proximal.
\item $d_\tau(J^o\cdot \rho(\gamma)_+,\rho(\gamma^{-1})_-)\geq 6r$ and $d_\tau(\rho(\gamma^{-1})_+,J^o\cdot \rho(\gamma)_-)\geq 6r$.
\item $d_\tau((J^o\rho(\gamma)J^o)_+,\rho(\gamma^{-1})_-)\geq 6r$ and $d_\tau(\rho(\gamma^{-1})_+,(J^o\rho(\gamma)J^o)_-)\geq 6r$.
\item The matrix $\rho(\gamma)$ belongs to $\mathscr{C}_{o,G}^>$ and the number 

\bc
$\vert b^o(\rho(\gamma))\vert^{\frac{1}{2}}   - \lambda_1(\rho(\gamma))$
\ec

\noindent is at distance at most $\delta$ from

\bc
$\frac{1}{2}\bb(J^o\cdot \rho(\gamma)_-,J^o\cdot \rho(\gamma)_+,\rho(\gamma^{-1})_-,\rho(\gamma^{-1})_+)$.
\ec

\item The number

\bc

$\vert b^\tau(\rho(\gamma))\vert^{\frac{1}{2}} - \lambda_1(\rho(\gamma)) $
\ec

\noindent is at distance at most $\delta$ from

\bc
$\frac{1}{2} \bb(J^o\cdot \rho(\gamma)_-,J^o\cdot \rho(\gamma)_+,\rho(\gamma^{-1})_-,\rho(\gamma^{-1})_+)-\frac{1}{2}\mathscr{G}_\tau(\rho(\gamma^{-1})_-,J^o\cdot \rho(\gamma)_+)$.
\ec

\end{enumerate}

\end{lema}

\begin{proof}

By transversality condition (\ref{eq transv condition}) there exists $r>0$ such that
\begin{equation}\label{eq transvversalidad A y B distribucion nu}
d_{\tau}(\xi(x),\eta(y))\geq 2r \textnormal{ and } d_{\tau}(\xi(y),\eta(x))\geq 2r
\end{equation}
\noindent for all $(x,y)\in A\times B$. Further, since $o\in\pmb{\Omega}_{\rho}$ we may assume
\begin{equation} \label{eq seisr en teo distribucion nu}
d_{\tau}(J^o\cdot\xi(x),\eta(x))\geq 6r
\end{equation}

\noindent for all $x\in\bg$. Given these $r>0$ and $2\delta>0$, we consider $\varepsilon>0$ as in Benoist's Theorem \ref{teo benoist}. 

By Lemma \ref{lema sambarino lemma 5.7} there exists a finite subset $F$ of $\gh$ outside of which elements satisfying $d_{\tau}(\rho(\gamma)_+,\rho(\gamma)_-)\geq 2r$ are $(r,\varepsilon)$-proximal. Thanks to (\ref{eq transvversalidad A y B distribucion nu}), for all $\gamma\in\gh\setminus F$ with $\gamma_-\in A$ and $\gamma_+\in B$ one has that $\rho(\gamma^{\pm 1})$ is $(r,\varepsilon)$-proximal. Moreover, since $J^o=(J^o)^{-1}$ preserves $\Vert\cdot\Vert_{\tau}$ we have that $J^o\rho(\gamma)J^o$ is $(r,\varepsilon)$-proximal with $(J^o\rho(\gamma)J^o)_{\pm}=J^o\cdot\rho(\gamma)_{\pm}$. In fact, by (\ref{eq seisr en teo distribucion nu}) we have

\bc
$d_{\tau}(J^o\cdot\rho(\gamma)_{+},\rho(\gamma^{-1})_{-})\geq 6r \textnormal{ and } d_{\tau}(\rho(\gamma^{-1})_{+},J^o\cdot\rho(\gamma)_{-})\geq 6r$.
\ec

\noindent Thanks to Proposition \ref{prop linear alg interpr of bo} (and Corollary \ref{cor gammao in cowmayor}), Proposition \ref{prop computing nu}, Theorem \ref{teo benoist} and the fact that $\lambda_1(\rho(\gamma^{-1}))$ equals $\lambda_1(\rho(\gamma))$ for all $\gamma$, the proof is finished.

\end{proof}

\subsection{The orbital counting functions of Theorems \ref{teorema A} and \ref{teorema B}}\label{subsec orbital counting functions}

\begin{prop}\label{prop counting with nu is well defined}

For every $t\geq 0$ one has

\bc
$\#\left\lbrace \gamma\in\Gamma: \hspace{0,3cm} \vert b^\tau(\rho(\gamma))\vert^{\frac{1}{2}} \leq t\right\rbrace<\infty$.
\ec
\end{prop}

\begin{proof}

By Remark \ref{rem btau is proper} the map $b^\tau$ descends to a proper map in $\hpq\cong G/H^o$, that we still denote by $b^\tau$. Hence

\bc
$C:=\lbrace o'\in\hpq:\hspace{0,3cm}  \vert b^\tau(o')\vert\leq t^2 \rbrace$
\ec

\noindent is compact. By Proposition \ref{prop action on omegarho is prop discont and limit set is limit set}, apart from possibly finitely many exceptions $\gamma$ in $\Gamma$, we have that $\rho(\gamma)\cdot o$ does not belong to $C$.

\end{proof}

The next proposition follows from a combination of Propositions \ref{prop linear alg interpr of bo} and \ref{prop computing nu}, Lemmas \ref{lema jrhojrho prox uno} and \ref{lema benoist proximal comp vasing y vap}, and the previous proposition.

\begin{prop}\label{prop counting with lambdauno is well defined}

For every $t\geq 0$ one has

\bc
$\#\left\lbrace \gamma\in\Gamma: \hspace{0,3cm} \rho(\gamma)\in\mathscr{C}_{o,G}^> \tn{ and } \vert b^o(\rho(\gamma))\vert^{\frac{1}{2}} \leq t\right\rbrace<\infty$.
\ec
\begin{flushright}
$\square$
\end{flushright}
\end{prop}

\begin{rem}\label{rem crit exponent coindes with the one of GM}
Assume that $\rho$ is $\hpq$-convex co-compact and the basepoint $o$ belongs to the convex hull of the limit set of $\rho$. By Corollary \ref{cor gammao in cowmayor} and Proposition \ref{prop ell dXG y vertboverto} we have that

\bc
$\displaystyle\limsup_{t\too\infty}\dfrac{\log\#\lbrace \gamma\in\Gamma: \hspace{0,3cm} \rho(\gamma)\in\mathscr{C}_{o,G}^> \tn{ and } \vert b^o(\rho(\gamma))\vert^{\frac{1}{2}}\leq t\rbrace}{t}$
\ec

\noindent coincides with

\bc
$\displaystyle\limsup_{t\too\infty}\dfrac{\log\#\lbrace \gamma\in\Gamma: \hspace{0,3cm} d_{\hpq}(o,\rho(\gamma)\cdot o)\leq t\rbrace}{t}$,
\ec

\noindent where $d_{\hpq}$ is the $\hpq$-distance introduced in \cite{GM}.

\begin{flushright}
$\diamond$
\end{flushright}
\end{rem}

\subsection{Weak triangle inequality}\label{subsec weak triangle}

The following proposition is inspired by \cite[Theorem 3.5]{GM}.

\begin{prop} \label{prop traingle inequality}
There exists a constant $L>0$ such that for every $f\in\Gamma$ there exists $D_f>0$ with the following property:  for every $\gamma\in\Gamma$  with $\vert\gamma\vert_\Gamma>L$ one has

\bc
$\frac{1}{2}\lambda_1(J^o\rho(f)\rho(\gamma)J^o\rho(\gamma^{-1})\rho(f^{-1}))\leq D_f+\frac{1}{2}\lambda_1(J^o\rho(\gamma)J^o\rho(\gamma^{-1}))$.
\ec

\end{prop}

We can think about the content of Proposition \ref{prop traingle inequality} as follows. Fix $f\in\Gamma$ such that $\rho(f)\in\mathscr{C}_{o,G}^>$. By Corollary \ref{cor gammao in cowmayor} for every $\gamma$ with $\vert\gamma\vert_\Gamma$ large enough one has $\rho(\gamma)\in\mathscr{C}_{o,G}^>$ and $\rho(f)\rho(\gamma)\in\mathscr{C}_{o,G}^>$. Thanks to Proposition \ref{prop ell dXG y vertboverto} and Proposition \ref{prop linear alg interpr of bo}, the inequality established in Proposition \ref{prop traingle inequality} can be stated as

\bc
$\ell_{o,\rho(f)\rho(\gamma)\cdot o}\leq D_f+\ell_{\rho(f)\cdot o,\rho(f)\rho(\gamma)\cdot o}$,
\ec

\noindent where the constant $D_f$ depends on the choice of $o$ and $f$ (and $\rho$) but not on the choice of $\gamma$. Even though the function $\ell_{\cdot,\cdot}$ is not a distance, we can heuristically think about $D_f$ as the term that replaces $\ell_{o,\rho(f)\cdot o}$ in the usual triangle inequality for distances.

\begin{proof}[Proof of Proposition \ref{prop traingle inequality}]

Take $0<\varepsilon\leq r$ as in Lemma \ref{lema jrhojrho prox uno}. Let $L>0$ such that for every $\gamma$ in $\Gamma$ with $\vert \gamma\vert_\Gamma >L$ the matrix $J^o\rho(\gamma)J^o\rho(\gamma^{-1})$ is $(r,\varepsilon)$-proximal. Fix $f\in\Gamma$ and let $\gamma$ be a element in $\Gamma$ with $\vert \gamma\vert_\Gamma >L$. We have

\bc
$\frac{1}{2}\lambda_1( J^o\rho(f)\rho(\gamma)J^o\rho(\gamma^{-1})\rho(f^{-1}))\leq\frac{1}{2}\log\Vert J^o\rho(f)\rho(\gamma)J^o\rho(\gamma^{-1})\rho(f^{-1})\Vert_\tau$.
\ec

\noindent By Remark \ref{rem J preserva norma} the right side number equals $\frac{1}{2}\log\Vert \rho(f)\rho(\gamma)J^o\rho(\gamma^{-1})\rho(f^{-1})\Vert_\tau$ which is less than or equal to

\bc
$  D_f'+\frac{1}{2}\log\Vert J^o\rho(\gamma)J^o\rho(\gamma^{-1})\Vert_\tau$
\ec

\noindent where $D_f':=\frac{1}{2}\log\Vert \rho(f)\Vert_\tau +\frac{1}{2}\log \Vert\rho(f^{-1})\Vert_\tau$. Since $J^o\rho(\gamma)J^o\rho(\gamma^{-1})$ is $(r,\varepsilon)$-proximal, we conclude by applying Lemma \ref{lema benoist proximal comp vasing y vap}.

\end{proof}

\section{Distribution of the orbit of $o$ with respect to $ b^o$}\label{section distrib wrt bo}
\setcounter{equation}{0}

In this section we prove Theorem \ref{teorema A}. The section is structured as follows: in Subsection \ref{subsec cocycle co} we define a Hölder cocycle on $\bg$ and the corresponding flow. In Subsection \ref{subsec dual and gromov co} we study the associated Gromov product. Theorem \ref{teorema A} in the torsion free case (resp. general case) is proved in Subsection \ref{subsec dist of fixed wrt bo} (resp. Subsection \ref{subsec proof teo A}).

For the rest of the section we fix $\rho:\Gamma\too G$ a $P_1^{p,q}$-Anosov representation and a point $o$  in $\pmb{\Omega}_\rho$.

\subsection{The cocycle $c_o$}\label{subsec cocycle co}

Observe that by definition of $\pmb{\Omega}_\rho$ and equivariance of the curves $\xi$ and $\eta$ the following map is well-defined.

\begin{dfn} \label{dfn cocycle o}

Let

\bc
$c_o:\Gamma\times\bg\too\rr: \hspace{0,3cm} c_o(\gamma,x):=\dfrac{1}{2}\log\left\vert\dfrac{\theta_{x}\left(\rho(\gamma^{-1})J^o\rho(\gamma)\cdot v_x\right)}{\theta_{x}\left(J^o\cdot v_x\right)}\right\vert$,
\ec

\noindent where $\theta_{x}:\rr^d\too\rr$ is a non-zero linear functional whose kernel equals $\eta(x)$ and $v_{x}\neq 0$ belongs to $\xi(x)$.
\end{dfn}

A geometric interpretation of the map $c_o$ is provided by the following remark. This characterization will not be used in the sequel.

\begin{rem} \label{rem co es la buseman de gm}

One can prove that for every $\gamma\in\Gamma$ and $x\in\bg$ one has

\bc
$c_{o}(\gamma,x)=\beta_{\xi(x)}(\rho(\gamma^{-1})\cdot o,o)$
\ec

\noindent where $\beta_\cdot(\cdot,\cdot)$ is the pseudo-Riemannian Busemann function defined by Glorieux-Monclair \cite[Definition 3.8]{GM}.
\begin{flushright}
$\diamond$
\end{flushright}
\end{rem}

Recall that a \textit{Hölder cocycle} is a function $c:\Gamma\times\bg\too\rr$ satisfying that for every $\gamma_0,\gamma_1$ in $\Gamma$ and $x\in\bg$ one has 

\bc
$c(\gamma_0\gamma_1,x)=c(\gamma_0,\gamma_1\cdot x)+c(\gamma_1,x)$
\ec

\noindent and such that the map $c(\gamma_0,\cdot)$ is Hölder (with the same exponent for every $\gamma_0$). The \textit{period} of (an infinite order element) $\gamma\in \gh$ is defined by $\ell_{c}(\gamma):=c(\gamma,\gamma_+)$.

\begin{lema}\label{lema cocycle o and periods}
The map $c_o$ is a Hölder cocycle. The period of $\gamma\in\gh$ is given by

\bc
$\ell_{c_o}(\gamma)=\lambda_1(\rho(\gamma))>0$.
\ec

\end{lema}

\begin{proof}

A direct computation shows that $c_o$ is a Hölder cocycle.

On the other hand let $\gamma\in\gh$ and fix a particular choice of a linear functional $\theta_{\gamma_+}$. Since $\lambda_1(\rho(\gamma))=\lambda_1(\rho(\gamma^{-1}))$ one sees that $\theta_{\gamma_+}\circ(\pm\rho(\gamma^{-1}))$ coincides with $ e^{\lambda_1(\rho(\gamma))}\theta_{\gamma_+}$ up to a sign (here $\pm\rho(\gamma^{-1})$ denotes some lift of $\rho(\gamma^{-1})$ to $\tn{SO}(p,q)$). The proof is now complete.

\end{proof}

Set $\bgc:=\lbrace(x,y)\in\bg\times\bg:\hspace{0,3cm} x\neq y\rbrace$ and consider the \textit{translation flow} on $\bgc\times \rr$ defined by
\begin{equation}\label{eq translation flow}
\psi_t(x,y,s):=(x,y,s-t).
\end{equation}
\noindent The group $\Gamma$ acts on $ \bgc\times\rr$ by
\begin{equation} \label{eq action via co}
\gamma\cdot(x,y,s):=(\gamma\cdot x,\gamma\cdot y, s-c_o(\gamma,y)).
\end{equation}
\noindent This action is proper and co-compact and we denote the quotient space by $\tn{U}_o\Gamma$. The flow $\psi_t$ descends to a flow on $\tn{U}_o\Gamma$, still denoted $\psi_t$, which is a Hölder reparametrization of the Gromov geodesic flow of $\Gamma$ \cite{Gro}. This is the analogue of Sambarino's Theorem \cite[Theorem 3.2(1)]{Sam} (see also Lemma \ref{lema conj urhogamma y uogamma}).

We say that an element $\gamma$ in $\Gamma$ is \textit{primitive} if cannot be written as a positive power of another element in $\Gamma$. Periodic orbits of $\psi_t$ are in one-to-one correspondence with conjugacy classes of primitive elements in $\Gamma$. If $[\gamma]$ is such a conjugacy class, the period of the corresponding periodic orbit is 

\bc
$\ell_{c_o}(\gamma)=\lambda_1(\rho(\gamma))$
\ec

\noindent (see Fact \ref{fact urhogamma is anosov} and Lemma \ref{lema conj urhogamma y uogamma}). The topological entropy of $\psi_t$ coincides with the \textit{entropy} of $\rho$ defined by Bridgeman-Canary-Labourie-Sambarino \cite{BCLS}:

\vspace{0,2cm}

\bc
$h_\tn{top}(\psi_t)=h_\rho:=\displaystyle\limsup_{t\too\infty} \dfrac{\log\#\lbrace [\gamma]\in [\Gamma]:\hspace{0,3cm} \gamma \tn{ is primitive and }\lambda_1(\rho(\gamma))\leq t\rbrace}{t}$.

\ec

\vspace{0,2cm}

\noindent It is positive and finite (c.f. Fact \ref{fact entropy, equidistribution and counting for urhogamma}) and will be denoted by $h$ from now on.

\begin{rem}
One can prove that if we \textit{push} all this construction by the limit map $\xi:\bg\too\Lambda_{\rho(\Gamma)}$, we recover the geodesic flow defined in \cite[Subsection 6.1]{GM} for $\hpq$-convex co-compact groups. This remark will not be used in the sequel.
\begin{flushright}
$\diamond$
\end{flushright}
\end{rem}

\subsection{Dual cocycle and Gromov product}\label{subsec dual and gromov co}

Thanks to transversality condition (\ref{eq transv condition}) and the fact that $o$ belongs to $\pmb{\Omega}_\rho$ the following map is well-defined.

\begin{dfn}
Let

\bc
$[\cdot,\cdot]_o:\bgc\too\rr: \hspace{0,3cm} [x,y]_o:=-\dfrac{1}{2}\log\left\vert  \dfrac{\theta_{x}\left(J^o\cdot v_x\right)\theta_{y}\left(J^o\cdot v_y\right)}{\theta_{x}\left(v_y\right)\theta_{y}\left(v_x\right)}\right\vert$,
\ec

\noindent where $\theta_x$ (resp. $\theta_y$) is a non-zero linear functional whose kernel is $\eta(x)$ (resp. $\eta(y)$) and $v_x$ (resp. $v_y$) is a non-zero vector in $\xi(x)$ (resp. $\xi(y)$).
\end{dfn}

\begin{rem}\label{rem gromov o is the one of GM}
The map $[\cdot,\cdot]_o$ coincides, up to a sign, with the Gromov product introduced in \cite[Subsection 3.5]{GM}. The authors give geometric interpretations of this function using pseudo-Riemannian geometry.
\begin{flushright}
$\diamond$
\end{flushright}
\end{rem}

\begin{rem}\label{rem coco are dual}
The cocycle $c_o$ is dual to itself, i.e. $\ell_{c_o}(\gamma)=\ell_{c_{o}}(\gamma^{-1})$ for every $\gamma\in\gh$. Indeed, this follows from Lemma \ref{lema cocycle o and periods} and the fact that $\lambda_1(g)=\lambda_1(g^{-1})$ for all $g$ in $ G$.
\begin{flushright}
$\diamond$
\end{flushright}
\end{rem}

The proof of the following lemma is a direct computation.

\begin{lema}\label{lema gromov o}
The map $[\cdot,\cdot]_o$ is a Gromov product for the pair $\lbrace c_o,c_o\rbrace$, that is, for every $\gamma\in\Gamma$ and every $(x,y)\in\bgc$ one has

\bc
$[\gamma\cdot x,\gamma\cdot y]_o-[x,y]_o=-(c_o(\gamma,x)+c_o(\gamma,y))$.
\ec
\begin{flushright}
$\square$
\end{flushright}
\end{lema}

The following lemma will be very important in the proof of Theorem \ref{teorema A}. It provides a geometric interpretation of the Gromov product different from the one given in Remark \ref{rem gromov o is the one of GM}.

\begin{lema}\label{lema computing gromov on gammapm o}
Let $\gamma$ be an element of $\gh$. Then

\bc
$[\gamma_-,\gamma_+]_o=-\frac{1}{2} \bb(J^o\cdot\rho(\gamma)_-,J^o\cdot\rho(\gamma)_+,\rho(\gamma^{-1})_-,\rho(\gamma^{-1})_+)$.
\ec
\end{lema}

\begin{proof}

From Section \ref{sec anosov} we know that $\rho(\gamma^{\pm 1})$ is proximal and that the following holds:

\bc
$\rho(\gamma)_+=\xi(\gamma_{+})$, \hspace{0,5cm} $\rho(\gamma^{-1})_+=\xi(\gamma_{-})$, \hspace{0,5cm} $\rho(\gamma)_-=\eta(\gamma_-)$, \hspace{0,5cm} $\rho(\gamma^{-1})_-=\eta(\gamma_{+})$.

\ec

\noindent Since $J^o=(J^o)^{-1}$, the matrix $J^o\rho(\gamma) J^o$ is proximal and one has the equalities

\bc
$(J^o\rho(\gamma) J^o)_{+}=J^o\cdot \xi(\gamma_+)$ and $(J^o\rho(\gamma) J^o)_{-}=J^o\cdot \eta(\gamma_-)$.
\ec

\noindent The proof finishes by a direct computation.

\end{proof}

\subsection{Distribution of attractors and repellors with respect to $ b^o$} \label{subsec dist of fixed wrt bo}

Recall that $h=h_{\tn{top}}(\psi_t)$ and let $\mu_{o}$ be a \textit{Patterson-Sullivan probability} on $\bg$ associated to $c_{o}$, i.e. $\mu_o$ satisfies

\bc
$\dfrac{d\gamma_*\mu_o}{d\mu_o}(x)=e^{-hc_{o}(\gamma^{-1},x)}$
\ec

\noindent for every $\gamma\in\Gamma$. Such a probability exists (see Subsection \ref{subsub PS}). By Lemma \ref{lema gromov o} the measure
\begin{equation}
e^{-h[\cdot,\cdot]_o}\mu_o\otimes\mu_o\otimes dt
\end{equation}
\noindent on $\bgc\times\rr$ is $\Gamma$-invariant and induces on the quotient $\tn{U}_o\Gamma$ a $\psi_t$-invariant measure. By Sambarino \cite[Theorem 3.2(2)]{Sam} this measure is, up to scaling, the probability of maximal entropy of $\psi_t$ (see Proposition \ref{prop product is of maximal entropy}).

For a metric space $X$ we denote by $C_c^*(X)$ the dual of the space of compactly supported continuous real functions on $X$ equipped with the weak-star topology. If $x$ is a point in $X$, let $\delta_x\in C_c^*(X)$ be the Dirac mass at $x$.

\begin{prop}[Sambarino {\cite[Proposition 4.3]{Sam}}\footnote{For a proof in our setting see Proposition \ref{prop distribution of periodic orbits}.}]\label{prop sambarino distribution wrt periods}

There exists a constant $M=M_{\rho,o}>0$ such that

\bc
$Me^{-ht}\displaystyle\sum_{\gamma\in\gh,\ell_{c_o}(\gamma)\leq t} \delta_{\gamma_-}\otimes\delta_{\gamma_+}\too e^{-h[\cdot,\cdot]_o}\mu_o\otimes\mu_o$
\ec

\noindent as $t\too\infty$ on $C_c^*(\bgc)$.
\begin{flushright}
$\square$
\end{flushright}
\end{prop}

From Proposition \ref{prop sambarino distribution wrt periods} we deduce Proposition \ref{prop distribution on bg for length} which directly implies Theorem \ref{teorema A} in the torsion free case.

Fix a point $\tau\in S^o$, a maximal subalgebra $\lieb\subset\liep^\tau\cap\lieq^o$ and a closed Weyl chamber $\lieb^+$ contained in $\lieb$.

\begin{prop} \label{prop distribution on bg for length}

There exists a  constant $M=M_{\rho,o}>0$ such that

\bc
$M e^{-ht}\displaystyle\sum_{\gamma\in\gh, \vert b^o(\rho(\gamma))\vert^{\frac{1}{2}}\leq t} \delta_{\gamma_-}\otimes\delta_{\gamma_+}\too\mu_o\otimes\mu_o$
\ec

\noindent as $t\too\infty$ on $C^*(\bg\times\bg)$.

\end{prop}

Recall that the generalized Cartan projection $b^o$ is defined in the set $\mathscr{C}_{o,G}^>$. The sum in Proposition \ref{prop distribution on bg for length} is taken then over all elements $\gamma\in\gh$ for which $\rho(\gamma)\in\mathscr{C}_{o,G}^>$ and $ \vert b^o(\rho(\gamma))\vert^{\frac{1}{2}}\leq t$. To make the formula more readable we do not emphasize the fact that $\rho(\gamma)$ must belong to $\mathscr{C}_{o,G}^>$. On the other hand, by Corollary \ref{cor gammao in cowmayor} this condition holds apart from finitely many exceptions $\gamma\in\Gamma$.

\begin{proof}[Proof of Proposition \ref{prop distribution on bg for length}]

Set 

\bc

$\theta_t:=M e^{-ht}\displaystyle\sum_{\gamma\in\gh,  \vert b^o(\rho(\gamma))\vert^{\frac{1}{2}}\leq t} \delta_{\gamma_-}\otimes\delta_{\gamma_+}$.
\ec

We first prove the statement outside the diagonal, that is, on subsets of $\bgc$. Let $\delta>0$ and $A,B\subset\bg$ disjoint open sets. Consider an element $\gamma\in\gh$ such that $\gamma_-\in A$ and $\gamma_+\in B$ and let $s:=[\gamma_-,\gamma_+]_{o}$. By taking $A$ and $B$ smaller we may assume 
\begin{equation} \label{eq aprox grom prod on dist thm nu}
\vert [x,y]_{o}-s\vert<\delta
\end{equation}
\noindent for all $(x,y)\in A\times B$.

By Lemma \ref{lema jrhojrho prox dos}, apart from possibly finitely many exceptions $\gamma\in \gh$ with $(\gamma_-,\gamma_+)\in A\times B$, the following holds:

\bc
$\left\vert  \vert b^o(\rho(\gamma))\vert^{\frac{1}{2}}-\lambda_1(\rho(\gamma))  -\frac{1}{2}\bb(J^o\cdot\rho(\gamma)_-,J^o\cdot \rho(\gamma)_+,\rho(\gamma^{-1})_-,\rho(\gamma^{-1})_+)\right\vert <\delta$.
\ec

\noindent Applying Lemma \ref{lema cocycle o and periods} and Lemma \ref{lema computing gromov on gammapm o} we conclude that

\bc
$\left\vert   \vert b^o(\rho(\gamma))\vert^{\frac{1}{2}}-\ell_{c_o}(\gamma)  +[\gamma_-,\gamma_+]_o\right\vert <\delta$.
\ec

\noindent By (\ref{eq aprox grom prod on dist thm nu}) it follows that

\bc
$ \ell_{c_{o}}(\gamma)-s-2 \delta< \vert b^o(\rho(\gamma))\vert^{\frac{1}{2}}<\ell_{c_{o}}(\gamma) -s+2 \delta$
\ec

\noindent holds apart from finitely many exceptions $\gamma\in\gh$ such that $\gamma_-\in A$ and $\gamma_+\in B$. From now on, the proof of the convergence

\bc
$\theta_t(A\times B)\too \mu_o(A)\mu_o(B)$
\ec

\noindent follows line by line the proof of \cite[Theorem 6.5]{Sam}.

It remains to prove the convergence in the diagonal $\lbrace(x,x): \hspace{0,3cm} x\in\bg\rbrace$, but once again, the proof is the same as the one given in \cite[Theorem 6.5]{Sam}. For completeness we briefly sketch it.

Since $\mu_o$ has no atoms (see Lemma \ref{lema PS non atomic}), for every $\gamma$ in $\Gamma$ the diagonal has $(\mu_o\otimes \gamma_*\mu_o)$-measure equal to zero. We fix two elements $\gamma_0,\gamma_1\in\gh$ with no common fixed point in $\bg$ and let $\varepsilon_0>0$. There exists a finite open covering $\mathscr{U}$ of $\bg$ such that for $i=0,1$ one has

\bc
$\displaystyle\sum_{U\in\mathscr{U}}  \mu_o(U) \mu_o(\gamma_i^{-1}\cdot U)<\varepsilon_0$.
\ec

\noindent We can assume that for every $U\in\mathscr{U}$ there exists $i\in\lbrace 0,1\rbrace$ such that $\gamma_i^{-1}\cdot\overline{U}$ is disjoint from $\overline{U}$. There exists an open covering $\mathscr{V}$ of $\bg$ with the following properties:

\begin{enumerate}

\item $\displaystyle\sum_{V\in\mathscr{V}}  \mu_o(V) \mu_o(\gamma_i^{-1} \cdot V)<\varepsilon_0$ for $i=0,1$.
\item The closure of every element in $\mathscr{U}$ is contained in a unique element of $\mathscr{V}$ and if $\gamma_i^{-1}\cdot\overline{U}$ is disjoint from $\overline{U}$ the same holds for this element in $\mathscr{V}$.
\item Suppose that $\gamma_i^{-1}\cdot\overline{U}\cap \overline{U}=\emptyset$ and let $V\in\mathscr{V}$ be the unique element such that $\overline{U}\subset V$. Then apart from finitely many exceptions $\gamma$ such that $\gamma_{\pm}\in U$ one has $(\gamma_i^{-1}\gamma)_-\in V$ and $(\gamma_i^{-1}\gamma)_+\in \gamma_i^{-1}\cdot V$.

\end{enumerate}

Set $D:=\displaystyle\max_{i=0,1}\lbrace D_{\gamma_i^{-1}}\rbrace$ where $D_{\gamma_i^{-1}}$ is the constant given by Proposition \ref{prop traingle inequality} and take $U\in\mathscr{U}$ as in (3). By Proposition \ref{prop traingle inequality} we have

\begin{equation*}
\begin{split}
\theta_t(U\times U) & \leq Me^{-ht}\displaystyle\sum_{\gamma\in\gh, \vert b^o(\rho(\gamma))\vert^{\frac{1}{2}}\leq t+D} \delta_{\gamma_-}(V)\delta_{\gamma_+}(\gamma_i^{-1}\cdot V)\\ & +Me^{-ht}\# F
\end{split}
\end{equation*}

\noindent where $F$ is a finite set independent of $t$. Since $V\times\gamma_i^{-1}\cdot V$ is far from the diagonal the right side converges to

\bc
$e^D\mu_o(V)\mu_o(\gamma_i^{-1}\cdot V)$
\ec

\noindent as $t\too \infty$. Adding up in $U\in\mathscr{U}$ we conclude

\bc
$\displaystyle\limsup_{t\too\infty}\displaystyle\sum_{U\in\mathscr{U}}\theta_t(U\times U)\leq 2e^D\varepsilon_0$.
\ec

\noindent Hence $\theta_t(\lbrace(x,x):\hspace{0,3cm} x\in\bg\rbrace)$ converges to zero and since the diagonal has measure zero for $\mu_o\otimes\mu_o$ the proof is finished.

\end{proof}

\subsection{Proof of Theorem \ref{teorema A}} \label{subsec proof teo A}

The following is a corollary of Proposition \ref{prop distribution on bg for length}.

\begin{cor}\label{cor distr orbit o in gammah for bo}
There exists a constant $M=M_{\rho,o}>0$ such that

\bc
$M e^{-ht}\displaystyle\sum_{\gamma\in\gh,  \vert b^o(\rho(\gamma))\vert^{\frac{1}{2}} \leq t} \delta_{\rho(\gamma^{-1})\cdot o^{\ppq}}\otimes\delta_{\rho(\gamma)\cdot o}\too \eta_{*}(\mu_o)\otimes\xi_*(\mu_o)$
\ec

\noindent on $C^*(\pp((\rr^d)^*)\times\pp(\rr^d))$ as $t\too\infty$.

\end{cor}

\begin{proof}

Set

\bc
$\nu_t^{\tn{H}}:= M e^{-ht}\displaystyle\sum_{\gamma\in\gh,  \vert b^o(\rho(\gamma))\vert^{\frac{1}{2}} \leq t} \delta_{\rho(\gamma^{-1})\cdot o^{\ppq}}\otimes\delta_{\rho(\gamma)\cdot o}$
\ec

\noindent and take $\theta_t$ the measure defined in the proof of Proposition \ref{prop distribution on bg for length}. We know that 

\bc
$(\eta,\xi)_*(\theta_t)\too \eta_{*}(\mu_o)\otimes\xi_*(\mu_o)$.
\ec

\noindent Hence we only have to show the following convergence
\begin{equation}\label{eq en cor orbit o para hiperbolicos}
\nu_t^{\tn{H}}-(\eta,\xi)_*(\theta_t)\too 0.
\end{equation}

Take a small positive $\delta$. By Proposition \ref{prop limit with S and U and Uuno cerca gammamas} and the proof of Proposition \ref{prop action on omegarho is prop discont and limit set is limit set} we know that, apart from finitely many exceptions $\gamma$ in $\gh$, one has

\bc
$d(\rho(\gamma)\cdot o,\rho(\gamma)_+)<\delta$ and $d(\rho(\gamma^{-1})\cdot o,\rho(\gamma^{-1})_+)<\delta$.
\ec

By taking $\cdot^{\ppq}$ we can assume further that $d^*(\rho(\gamma^{-1})\cdot o^{\ppq},\rho(\gamma)_-)<\delta$. Now the proof of (\ref{eq en cor orbit o para hiperbolicos}) follows from evaluation on continuous functions of $\pp((\rr^d)^*)\times\pp(\rr^d)$.

\end{proof}

We now include torsion elements to the previous statement and finish the proof of Theorem \ref{teorema A}.

\begin{prop}\label{prop distribution on bg for length with torsion}

There exists a constant $M=M_{\rho,o}>0$ such that

\bc
$M e^{-ht}\displaystyle\sum_{\gamma\in\Gamma,  \vert b^o(\rho(\gamma))\vert^{\frac{1}{2}} \leq t} \delta_{\rho(\gamma^{-1})\cdot o^{\ppq}}\otimes\delta_{\rho(\gamma)\cdot o}\too \eta_{*}(\mu_o)\otimes\xi_*(\mu_o)$
\ec

\noindent on $C^*(\pp((\rr^d)^*)\times\pp(\rr^d))$ as $t\too\infty$.
\end{prop}

\begin{proof}

The structure of the proof is the same as that of Proposition \ref{prop distribution on bg for length}, that is, we first prove the statement outside the diagonal and deduce from that the statement on the diagonal. Here by \textit{diagonal} we mean the set

\bc
$\Delta:=\lbrace (\theta,v)\in \pp((\rr^d)^*)\times\pp(\rr^d): \hspace{0,3cm} \theta(v)=0 \rbrace$.
\ec

Let 

\bc
$\nu_t:= M e^{-ht}\displaystyle\sum_{\gamma\in\Gamma,  \vert b^o(\rho(\gamma))\vert^{\frac{1}{2}} \leq t} \delta_{\rho(\gamma^{-1})\cdot o^{\ppq}}\otimes\delta_{\rho(\gamma)\cdot o}$
\ec

\noindent and take $\nu_t^{\tn{H}}$ as in the proof of Corollary \ref{cor distr orbit o in gammah for bo}.

Consider first a continuous function $f$ on $\pp((\rr^d)^*)\times\pp(\rr^d)$ whose support $\tn{supp}(f)$ is disjoint from $\Delta$.

\begin{cla} \label{claim f suppor far from diagonal implies g in gh}

The following holds

\bc
$\#\lbrace \gamma\in\Gamma: \hspace{0,3cm} (\rho(\gamma^{-1})\cdot o^{\ppq},\rho(\gamma)\cdot o)\in \tn{supp}(f) \tn{ and } \gamma\notin\gh\rbrace<\infty$.
\ec

\end{cla}

\begin{proof}[Proof of Claim \ref{claim f suppor far from diagonal implies g in gh}]

Fix a positive $D$ such that for every $(\theta,v)\in \tn{supp}(f)$ one has $d(\theta,v)> D$. As we saw in the proof of Proposition \ref{prop action on omegarho is prop discont and limit set is limit set}, the distances

\bc
$d(\rho(\gamma)\cdot o,U_1(\rho(\gamma)))$ and $d^*(\rho(\gamma^{-1})\cdot o^{\ppq},S_{d-1}(\rho(\gamma)))$
\ec

\noindent converge to zero as $\gamma\too\infty$. We conclude that, apart from possibly finitely many exceptions $\gamma$ in $\Gamma$ with $(\rho(\gamma^{-1})\cdot o^{\ppq},\rho(\gamma)\cdot o)\in \tn{supp}(f)$, one has

\bc
$d(U_1(\rho(\gamma)),S_{d-1}(\rho(\gamma)))>D$.
\ec

Now apply (\ref{eq def anosov}), Remark \ref{rem complemento de Sdmenosuno va en Uuno} and Benoist's Lemma \ref{lema benoist lemma 1.2} to conclude that for $\vert\gamma\vert_\Gamma$ large enough the matrix $\rho(\gamma)$ is proximal.

\end{proof}

From Claim \ref{claim f suppor far from diagonal implies g in gh} we conclude that 

\bc
$\displaystyle\lim_{t\too\infty}\nu_t(f)=\displaystyle\lim_{t\too\infty}\nu_t^{\tn{H}}(f)$
\ec

\noindent which by Corollary \ref{cor distr orbit o in gammah for bo} equals $(\eta_{*}(\mu_o)\otimes\xi_*(\mu_o))(f)$.

It remains to prove the convergence on the diagonal. It suffices to prove that for every positive $\varepsilon_0$ there exists an open covering $\lbrace U^*\times U\rbrace$ of $\Delta$ such that

\bc
$\displaystyle\limsup_{t\too\infty}\nu_t\left(\displaystyle\bigcup (U^*\times U)\right)\leq \varepsilon_0$.
\ec

The proof is the same as in Proposition \ref{prop distribution on bg for length}. Namely, take two elements $\gamma_0,\gamma_1$ in $\gh$ with no common fixed point in $\bg$ and a coverings $\mathscr{U}=\lbrace U^*\times U\rbrace$ and $\mathscr{V}=\lbrace V^*\times V\rbrace$ of $\Delta$ by open sets with the following properties:

\begin{enumerate}
\item For every $U^*\times U$ in $\mathscr{U}$ there exists $i=0,1$ such that $\rho(\gamma_i^{-1})\cdot\overline{U}$ is transverse to $\overline{U^*}$.
\item $\displaystyle\sum_{V^*\times V\in\mathscr{V}} (\eta_{*}(\mu_o)\otimes\xi_*(\mu_o))(V^*\times\rho(\gamma_i^{-1})\cdot V)<\varepsilon_0$ for $i=0,1$.
\item The closure of every element in $\mathscr{U}$ is contained in a unique element of $\mathscr{V}$ and if $\rho(\gamma_i^{-1})\cdot \overline{U}$ is transverse to $\overline{U^*}$ the same holds for this element in $\mathscr{V}$.
\item Suppose that $\rho(\gamma_i^{-1})\cdot \overline{U}$ is transverse to $\overline{U^*}$ and let $V^*\times V\in\mathscr{V}$ be the unique element such that $\overline{U}\subset V$ and $\overline{U^*}\subset V^*$. Then, apart from possibly finitely many exceptions $\gamma$ such that $(\rho(\gamma^{-1})\cdot o^{\ppq},\rho(\gamma)\cdot o)\in U^*\times U$, one has

\bc
$(\rho((\gamma_i^{-1}\gamma)^{-1})\cdot o^{\ppq},\rho(\gamma_i^{-1}\gamma)\cdot o)\in V^*\times \rho(\gamma_i^{-1})\cdot V$.
\ec
\end{enumerate}

Provided with this construction, the proof finishes in the same way as that of Proposition \ref{prop distribution on bg for length}.

\end{proof}

\begin{rem}\label{rem crit exponent coincides with the entropy}
From Proposition \ref{prop distribution on bg for length with torsion} we deduce that

\bc
$\displaystyle\lim_{t\too\infty}\dfrac{\log\#\lbrace \gamma\in\Gamma: \hspace{0,3cm} \rho(\gamma)\in\mathscr{C}_{o,G}^> \tn{ and } \vert b^o(\rho(\gamma))\vert^{\frac{1}{2}}\leq t\rbrace}{t}$
\ec

\noindent coincides with the entropy $h=h_\rho$ of $\rho$.
\begin{flushright}
$\diamond$
\end{flushright}
\end{rem}

\section{Distribution of the orbit of $o$ with respect to $ b^\tau$}\label{section distrib wrt btau}
\setcounter{equation}{0}

The proof of Theorem \ref{teorema B} follows the same lines of the proof of Theorem \ref{teorema A}, we just have to pick a (slightly) different flow $\psi_t$.

Fix a $P_1^{p,q}$-Anosov representation $\rho:\Gamma\too G$, a point $o$ in $\pmb{\Omega}_\rho$ and $\tau\in S^o$.

\subsection{The cocycle $c_\tau$}\label{subsec cocycle ctau}

Let $\Vert\cdot\Vert_\tau$ be the norm introduced in Subsection \ref{subsec cartan dec in sec anosov}.

\begin{dfn} \label{dfn cocycle tauo}

Let

\bc
$c_{\tau}:\Gamma\times\bg\too\rr: \hspace{0,3cm} c_{\tau}(\gamma,x):=\dfrac{1}{2}\log\left(\dfrac{\Vert\rho(\gamma)\cdot\theta_{x}\Vert_{\tau}\Vert\rho(\gamma)\cdot v_{x}\Vert_{\tau}}{\Vert\theta_{x}\Vert_{\tau}\Vert v_{x}\Vert_{\tau}}\right)$
\ec

\noindent where $\theta_{x}:\rr^d\too\rr$ is a non-zero linear functional whose kernel equals $\eta(x)$ and $v_{x}\neq 0$ belongs to $\xi(x)$.

\end{dfn}

\begin{rem} \label{rem ctau es el betauno}
One can prove that for every $\gamma\in\Gamma$ and $x\in\bg$ one has

\bc
$c_{\tau}(\gamma,x)=\log\dfrac{\Vert\rho(\gamma)\cdot v_{x}\Vert_{\tau}}{\Vert v_{x}\Vert_{\tau}}$,
\ec

\noindent that is, $c_\tau$ coincides with the map $\beta_1(\cdot,\cdot)$ of \cite[Section 5]{Sam}. This remark will not be used in the sequel.
\begin{flushright}
$\diamond$
\end{flushright}
\end{rem}

The following lemma holds by straightforward computations.

\begin{lema}\label{lema cocycle tau and periods}

The function $c_{\tau}$ is a Hölder cocycle. The period of $\gamma$ in $\gh$ is given by

\bc  
$\ell_{c_{\tau}}(\gamma)=\lambda_1(\rho(\gamma))>0$.
\ec
\begin{flushright}
$\square$
\end{flushright}
\end{lema}

The quotient space of $\bgc\times \rr$ by the action of $\Gamma$ induced by $c_\tau$ will be denoted by $\tn{U}_\tau\Gamma$. It is equipped with a flow that lifts to the translation flow (\ref{eq translation flow}) on $\bgc\times\rr$.

\subsection{Dual cocycle and Gromov product}\label{subsec dual and gromov ctau}

\begin{dfn}

Let
\bc
$[\cdot,\cdot]_{\tau}:\bgc\too\rr: \hspace{0,3cm} [x,y]_{\tau}:=\dfrac{1}{2}\log\left\vert  \dfrac{\theta_{y}\left(v_x\right)\theta_{x}\left(v_y\right)}{\theta_{x}\left(J^o\cdot v_x\right)\Vert\theta_{y}\Vert_{\tau}\Vert v_y\Vert_{\tau}} \right\vert$.
\ec
\end{dfn}

\begin{rem} \label{rem co dual to ctauo} 
Recall that $c_o$ is the cocycle defined in Section \ref{section distrib wrt bo}. The cocycle $c_{\tau}$ is dual to $c_o$, i.e. $\ell_{c_o}(\gamma)=\ell_{c_{\tau}}(\gamma^{-1})$ for every $\gamma\in\gh$.
\begin{flushright}
$\diamond$
\end{flushright}
\end{rem}

The proof of the following lemma is a direct computation.

\begin{lema}
For every $\gamma\in\Gamma$ and every $(x,y)\in\bgc$ one has

\bc
$[\gamma\cdot x,\gamma\cdot y]_{\tau}-[x,y]_{\tau}=-(c_o(\gamma,x)+c_{\tau}(\gamma,y))$.
\ec
\begin{flushright}
$\square$
\end{flushright}
\end{lema}

\begin{lema}\label{lema computing gromov tauo on gammapm}
Let $\gamma$ be an element of $\gh$. Then

\bc
$[\gamma_-,\gamma_+]_{\tau}=-\frac{1}{2}\bb(J^o\cdot \rho(\gamma)_-,J^o\cdot\rho(\gamma)_+,\rho(\gamma^{-1})_-,\rho(\gamma^{-1})_+)+\frac{1}{2}\mathscr{G}_{\tau}(\rho(\gamma^{-1})_-,J^o\cdot \rho(\gamma)_+)$.
\ec
\end{lema}

\begin{proof}

Recall the definition of $[\cdot,\cdot]_o$ from Subsection \ref{subsec dual and gromov co}. One has

\bc

$[\gamma_-,\gamma_+]_\tau=[\gamma_-,\gamma_+]_o+\dfrac{1}{2}\log\dfrac{\left\vert \theta_{\gamma_+}(J^o\cdot v_{\gamma_+}) \right\vert}{\Vert \theta_{\gamma_+} \Vert_\tau\Vert  v_{\gamma_+}\Vert_\tau} $.

\ec
\noindent The proof then follows from Lemma \ref{lema computing gromov on gammapm o} and Remark \ref{rem J preserva norma}.

\end{proof}

\subsection{Distribution of attractors and repellors with respect to $ b^\tau$} \label{subsec dist of fixed wrt btau}

Let $\mu_{\tau}$ be a Patterson-Sullivan probability on $\bg$ associated to $c_{\tau}$ and recall that $\mu_o$ is the one associated to $c_o$. The analogue of Proposition \ref{prop sambarino distribution wrt periods} is available for the flow on $\tn{U}_\tau\Gamma$. The limit measure can be written in this case as\footnote{For a proof, see Remark \ref{rem BM for ctau and distribution}.}

\bc
$e^{-h[\cdot,\cdot]_\tau}\mu_o\otimes\mu_\tau $.
\ec

Let $\lieb^+$ be a closed Weyl chamber of a maximal subalgebra $\lieb\subset\liep^\tau\cap\lieq^o$.

\begin{prop} \label{prop distribution on bg nu}
There exists a constant $M'=M'_{\rho,\tau}>0$ such that

\bc
$M' e^{-ht}\displaystyle\sum_{\gamma\in\gh, \vert b^\tau(\rho(\gamma))\vert^{\frac{1}{2}}\leq t} \delta_{\gamma_-}\otimes\delta_{\gamma_+}\too \mu_o\otimes\mu_\tau$
\ec

\noindent as $t\too\infty$ on $C^*(\bg\times\bg)$.

\end{prop}

\begin{proof}

The proof is the same that the one given in Proposition \ref{prop distribution on bg for length} adapted to the pair $\lbrace c_o,c_\tau \rbrace$ and the Gromov product $[\cdot,\cdot]_\tau$: apply item (5) of Lemma \ref{lema jrhojrho prox dos} and Lemma \ref{lema computing gromov tauo on gammapm}.

\end{proof}

\subsection{Proof of Theorem \ref{teorema B}} \label{subsec proof teo B}

The following proposition, which implies Theorem \ref{teorema B}, can be proved in the same way as Proposition \ref{prop distribution on bg for length with torsion}.

\begin{prop}\label{prop distribution on bg for btau with torsion}
There exists a constant $M'=M'_{\rho,\tau}>0$ such that

\bc
$M' e^{-ht}\displaystyle\sum_{\gamma\in\Gamma, \vert b^\tau(\rho(\gamma))\vert^{\frac{1}{2}} \leq t} \delta_{\rho(\gamma^{-1})\cdot o^{\ppq}}\otimes\delta_{\rho(\gamma)\cdot o}\too \eta_{*}(\mu_o)\otimes\xi_*(\mu_\tau)$
\ec

\noindent on $C^*(\pp((\rr^d)^*)\times\pp(\rr^d))$ as $t\too\infty$.
\begin{flushright}
$\square$
\end{flushright}
\end{prop}

\appendix

\section{Distribution of periodic orbits in $\tn{U}_o\Gamma$ and $\tn{U}_\tau\Gamma$} 
\label{appendix distribution utaugamma y uogamma}
\setcounter{equation}{0}

The goal of this appendix is to describe the distribution of periodic orbits of the flows defined in Sections \ref{section distrib wrt bo} and \ref{section distrib wrt btau} (Proposition \ref{prop distribution of periodic orbits} and Remark \ref{rem BM for ctau and distribution}). For the case on which $\Gamma$ is the fundamental group of a closed negatively curved manifold, this result is covered by \cite[Proposition 4.3]{Sam}. Here we treat the case of word hyperbolic groups admitting an Anosov representation.

In \cite[Proposition 4.3]{Sam} the author applies the thermodynamic formalism to re\-para\-me\-tri\-za\-tions of the geodesic flow of the manifold. Here we benefit from the fact that a projective Anosov representation $\rho$ is given and use the \textit{geodesic flow} of $\rho$, introduced by Bridgeman-Canary-Labourie-Sambarino in \cite{BCLS}, as a reference flow. This is a canonical flow associated to a projective Anosov representation  and we show that it is Hölder conjugate to the flows on the spaces $\tn{U}_o\Gamma$ and $\tn{U}_\tau\Gamma$. Since the techniques of the thermodynamic formalism are available for the geodesic flow of the representation (see \cite{BCLS,CLT}), the adaptations needed in our context are straightforward.

The appendix is structured as follows. In Subsection \ref{subsec geod flow def and prop} we recall the definition of the geodesic flow of a representation and its main properties. We are interested in two descriptions of its probability of maximal entropy (Facts \ref{fact entropy, equidistribution and counting for urhogamma} and \ref{fact product measure is BM}). In Subsection \ref{subsec flows uogamma y utaugamma} we translate these results to the flows on $\tn{U}_o\Gamma$ and $\tn{U}_\tau\Gamma$.

\subsection{The geodesic flow $\tn{U}_\rho\Gamma$}\label{subsec geod flow def and prop}

We fix from now on a projective Anosov representation $\rho:\Gamma\too G$.

\subsubsection{\tn{\textbf{Definition and the metric Anosov property}}}\label{subsub geod flow of rho is metric anosov}

The standard reference for this subsection is \cite{BCLS}. Given $(x,y)\in\bgc$ let

\bc
$\tn{M}(x,y):=\lbrace(\theta,v)\in \eta(x)\times\xi(y): \hspace{0,3cm} \theta(v)=1\rbrace/\sim$
\ec

\noindent where $(\theta,v)\sim(-\theta,-v)$. Consider the line bundle over $\bgc$ defined by

\bc
$F_\rho:=\lbrace (x,y,\theta,v): \hspace{0,3cm} (x,y)\in\bgc \tn{ and } (\theta,v)\in  \tn{M}(x,y)\rbrace$.
\ec

\begin{hecho}[Bridgeman-Canary-Labourie-Sambarino {\cite[Sections 4 \& 5]{BCLS}}]\label{fact urhogamma is anosov}
The following holds:

\begin{itemize}

\item The group $\Gamma$ acts naturally on $F_\rho$ and this action is proper and co-compact. The quotient space is denoted by $\tn{U}_\rho\Gamma$.

\item The flow $\phi_t$ on $F_\rho$ defined by

\bc
$\phi_t(x,y,\theta,v):=(x,y,e^{-t}\theta,e^t v)$
\ec

\noindent descends to a flow on $\tn{U}_\rho\Gamma$, still denoted by $\phi_t$, and called the \textit{geodesic flow} of $\rho$. The geodesic flow of $\rho$ is conjugate, by a Hölder homeomorphism, to a Hölder reparametrization of the Gromov geodesic flow of $\Gamma$ (see Mineyev \cite{Min}).

\item Periodic orbits of $\phi_t$ are in one-to-one correspondence with conjugacy classes of primitive elements $\gamma$ in $\Gamma$. The corresponding period is $\lambda_1(\rho(\gamma))$. 

\item The geodesic flow $\phi_t$ is a transitive \textit{metric Anosov flow}. Very informally, this means that there exists laminations $W^{ss}$, $W^{uu}$, $W^{cs}$ and $W^{cu}$ of $\tn{U}_\rho\Gamma$, called respectively \textit{strong stable lamination}, \textit{strong unstable lamination}, \textit{central stable lamination} and \textit{central unstable lamination}, defining a \textit{local product structure} and with the property that $W^{ss}$ (resp. $W^{uu}$) is exponentially contracted by the flow (resp. the inverse flow). For precise definitions see \cite[Subsection 3.2]{BCLS}.

Explicitly, for a point $Z_0=(x_0,y_0,\theta_0,v_0)$ in $\tn{U}_\rho\Gamma$ the strong stable and strong unstable leaves through $Z_0$ are given by:

\bc
$W^{ss}(Z_0)=\lbrace (x,y_0,\theta,v_0)\in\tn{U}_\rho\Gamma:\hspace{0,3cm} \theta\in\eta(x) \tn{ and } \theta(v_0)=1 \rbrace$
\ec

\noindent and

\bc
$W^{uu}(Z_0)=\lbrace (x_0,y,\theta_0,v)\in\tn{U}_\rho\Gamma:\hspace{0,3cm} v\in\xi(y) \tn{ and } \theta_0(v)=1\rbrace$.
\ec

\noindent The central stable and central unstable leaves are given by:

\bc
$W^{cs}(Z_0)=\lbrace (x,y_0,\theta,v)\in\tn{U}_\rho\Gamma:\hspace{0,3cm} \theta\in\eta(x),\hspace{0,1cm} v\in\xi(y_0) \tn{ and } \theta(v)=1 \rbrace$
\ec

\noindent and

\bc
$W^{cu}(Z_0)=\lbrace (x_0,y,\theta,v)\in\tn{U}_\rho\Gamma:\hspace{0,3cm} v\in\xi(y),\hspace{0,1cm} \theta\in\eta(x_0) \tn{ and } \theta(v)=1\rbrace$.
\ec

\end{itemize}
\begin{flushright}
$\diamond$
\end{flushright}
\end{hecho}

\subsubsection{\tn{\textbf{Entropy and distribution of periodic orbits}}}\label{subsub entropy and distribution of periodic orbits}

A flow is said to be \textit{topologically weakly-mixing} if all the periods of its periodic orbits are not multiple of a common constant.

\begin{prop}\label{prop geod flow is weak mixing}
The geodesic flow of $\rho$ is topologically weakly-mixing.
\end{prop}

Before proving Proposition \ref{prop geod flow is weak mixing} let us state the main result of this subsection. Indeed, the following fact is a consequence of the existence of a strong Markov coding for $\phi_t$ (see \cite{BCLS,CLT}) together with the weak-mixing property. For Axiom A flows it was originally proved by Bowen \cite{Bow1} (the counting result is due to Parry-Pollicott \cite{PP}). In order to obtain it in our more general context, we need to apply Pollicott's work \cite[Subsection 3.5]{Pol}.

\begin{hecho}\label{fact entropy, equidistribution and counting for urhogamma}
The following holds:

\begin{itemize}

\item 
The topological entropy of $\phi_t$ is positive and finite. It is given by

\bc
$h=h_\rho:=\displaystyle\limsup_{t\too\infty}\dfrac{\log\#\lbrace[\gamma]\in[\Gamma]:\hspace{0,3cm} \gamma \tn{ is primitive and } \lambda_1(\rho(\gamma))\leq t\rbrace}{t}$.
\ec

\item As $t\too\infty$, one has

\bc
$hte^{-ht}\#\lbrace[\gamma]\in[\Gamma]:\hspace{0,3cm}  \gamma \tn{ is primitive and } \lambda_1(\rho(\gamma))\leq t\rbrace\too 1$.
\ec

\item There exists a unique probability $m=m_\rho$ of maximal entropy for $\phi_t$, called the \textit{Bowen-Margulis probability}.

\item Periodic orbits become equidistributed with respect to $m$: if $\tn{Leb}_{[\gamma]}$ denotes the Lebesgue measure of length $\lambda_1(\rho(\gamma))$ supported on the periodic orbit $[\gamma]$, then

\bc
$hte^{-ht}\displaystyle\sum \frac{1}{\lambda_1(\rho(\gamma))}\tn{Leb}_{[\gamma]}\too m$
\ec

\noindent in the weak-star topology as $t\too \infty$. Here the sum is taken over all conjugacy classes of primitive elements $\gamma$ such that $\lambda_1(\rho(\gamma))\leq t$.

\end{itemize}

\end{hecho}

\begin{flushright}
$\diamond$
\end{flushright}

We finish this subsection with an elementary proof of Proposition \ref{prop geod flow is weak mixing} inspired by the work of Benoist \cite{Ben1}.

\begin{proof}[Proof of Proposition \ref{prop geod flow is weak mixing}]
Suppose by contradiction that $\phi_t$ is not topologically weakly-mixing. By Fact \ref{fact urhogamma is anosov} this implies that there exists a constant $a>0$ such that the group spanned by  the set $\lbrace\lambda_1(\rho(\gamma))\rbrace_{\gamma\in\Gamma}$ is contained in $a\zz$. 

Set 

\bc
$\bgcu:=\lbrace(x_1,x_2,x_3,x_4)\in(\bg)^4: \hspace{0,3cm} (x_i,x_j)\in\bgc \tn{ for all } i\neq j  \rbrace$.
\ec

\noindent Since $\lbrace(\gamma_-,\gamma_+)\rbrace_{\gamma\in\gh}$ is dense in $\bgc$ (see Gromov \cite[Corollary 8.2.G]{Gro}), Benoist's Theorem \ref{teo benoist} implies that
\begin{equation}\label{eq crosstrio in azz}
\lbrace\bb(\eta(x'),\xi(y'),\eta(x),\xi(y)):\hspace{0,3cm} (x',y',x,y)\in\bgcu \rbrace\subset a\zz.
\end{equation}

Fix three different points $x',y'$ and $y$ in $\bg$. Transversality condition (\ref{eq transv condition}) and the definition of the cross-ratio implies the following: for every $x\in\bg$ such that $(x',y',x,y)\in\bgcu$ there exists a neighbourhood $V$ of $x$ and a point $\xi_{x,y,y'}$ in the projective line $\xi(y)\oplus\xi(y')$ such that
\begin{equation}\label{eq geod flow is weak mixing}
\eta(\tilde{x})\cap (\xi(y)\oplus\xi(y'))=\lbrace \xi_{x,y,y'}\rbrace
\end{equation}
\noindent holds for every $\tilde{x}\in V$.

\begin{cla} \label{claim in geod flow is weak mix}
Assume that (\ref{eq geod flow is weak mixing}) holds. Then the limit set $\Lambda_{\rho(\Gamma)}$ is not contained in $\xi(y)\oplus\xi(y')$.
\end{cla}

\begin{proof}[Proof of Claim \ref{claim in geod flow is weak mix}]

Suppose by contradiction that $\Lambda_{\rho(\Gamma)}\subset\xi(y)\oplus\xi(y')$. Transversality condition (\ref{eq transv condition}) implies that for every $x\in\bg$ different from $y'$ and $y$ one has

\bc
$\eta(x)\cap (\xi(y)\oplus\xi(y'))=\lbrace \xi(x)\rbrace$.
\ec

\noindent Then by (\ref{eq geod flow is weak mixing}) the map $\xi$ is not injective and this is a contradiction.

\end{proof}

Because of Claim \ref{claim in geod flow is weak mix} we can take $y''$ in $\bg$ such that $\xi(y'')$ does not belong to $\xi(y)\oplus\xi(y')$. We can assume further that $y''\neq x'$.

By (\ref{eq crosstrio in azz}) we have again the following: for every $x\notin\lbrace x',y,y',y'' \rbrace$ there exists a neighbourhood $V$ of $x$ and a point $\xi_{x,y,y''}$ in the projective line $\xi(y)\oplus\xi(y'')$ such that

\bc
$\eta(\tilde{x})\cap (\xi(y)\oplus\xi(y''))=\lbrace \xi_{x,y,y''}\rbrace$
\ec

\noindent holds for every $\tilde{x}\in V$.

As in Claim \ref{claim in geod flow is weak mix} we conclude that $\Lambda_{\rho(\Gamma)}$ cannot be contained in $\xi(y)\oplus\xi(y')\oplus\xi(y'')$ and now an inductive argument yields the desired contradiction.

\end{proof}

\subsubsection{\tn{\textbf{The invariant measure of the strong stable lamination}}}\label{subsub transverse measure}

As shown by Margulis \cite{Mar2}, for Anosov flows there exists an invariant measure of the strong stable lamination which is exponentially contracted by the flow. In our context this measure is also available: this follows from the thermodynamic formalism as explained by Bowen-Marcus \cite[Section 4]{BM}. As we shall see in Fact \ref{fact product measure is BM}, the importance for us of this measure relies on the fact that describes the probability of maximal entropy of $\phi_t$ in a different way that the one provided by Fact \ref{fact entropy, equidistribution and counting for urhogamma}.

The statement that we need is the following (for precisions see \cite{BM}).

\begin{hecho}\label{fact existence of transverse measure}
Given any $Z_0\in\tn{U}_\rho\Gamma$ and any small relative neighbourhood $W^{cu}_{\tn{loc}}(Z_0)$ of $Z_0$ in $W^{cu}(Z_0)$, there exists a positive and finite Borel measure $\nu^{cu}_{\tn{loc}}(Z_0)$ on $W^{cu}_{\tn{loc}}(Z_0)$ such that:

\begin{itemize}
\item The family $\lbrace  \nu^{cu}_{\tn{loc}}(Z_0)\rbrace_{Z_0\in\tn{U}_\rho\Gamma}$ is $W^{ss}$-invariant\footnote{The precise definition of a $W^{ss}$-\textit{invariant measure} can be found in \cite[p.43]{BM}. Very informally, this means that if we have a map between two local leaves $W^{cu}_{\tn{loc}}(Z_0)$ and $W^{cu}_{\tn{loc}}(Z_1)$ which is defined \textit{following the leaves of $W^{ss}$}, then this map sends the measure $\nu^{cu}_{\tn{loc}}(Z_0)$ to the measure $\nu^{cu}_{\tn{loc}}(Z_1)$.}.

\item There exists a real number $h^u \geq 0$ such that for every $t$ and every $Z_0\in\tn{U}_\rho\Gamma$ one has

\bc
$(\phi_t)_*(\nu^{cu}_{\tn{loc}}(Z_0))=e^{-h^ut}\nu^{cu}_{\tn{loc}}(\phi_t(Z_0))$.
\ec
\end{itemize}
\begin{flushright}
$\diamond$
\end{flushright}
\end{hecho}

\subsubsection{\tn{\textbf{The Bowen-Margulis probability}}}\label{subsub bowenmargulis for Urhogamma}
By reversing time and disintegrating along flow lines, Fact \ref{fact existence of transverse measure} yields a family of measures $\lbrace \nu^{ss}_{\tn{loc}}(Z_0) \rbrace$ on local strong stable leaves which is expanded by the flow. In the case of Anosov flows, Margulis \cite{Mar2} first showed how the families $\lbrace \nu^{cu}_{\tn{loc}}(Z_0) \rbrace$ and $\lbrace \nu^{ss}_{\tn{loc}}(Z_0) \rbrace$ with the above properties combine to produce a $\phi_t$-invariant finite Borel measure $\nu$ in the whole space. This measure coincides, up to scaling, with the Bowen-Margulis probability of the flow.

The statement that we need in our context is the following. Once again, this is a standard fact and the reader is referred for instance to Katok-Hasselblatt's book \cite[Section 5 of Chapter 20]{KH} for a proof in the case of Anosov flows. With obvious adaptations the same proof works in our setting.

\begin{hecho}\label{fact product measure is BM}

Suppose that one has a family of measures $\lbrace \nu^{ss}_{\tn{loc}}(Z_0) \rbrace_{Z_0\in\tn{U}_\rho\Gamma}$ on the local strong stable leaves with the following properties:

\begin{itemize}
\item There exists a real number $h^s \geq 0$ such that for every $t$ and every $Z_0\in\tn{U}_\rho\Gamma$ one has

\bc
$(\phi_t)_*(\nu^{ss}_{\tn{loc}}(Z_0))=e^{h^st}\nu^{ss}_{\tn{loc}}(\phi_t(Z_0))$.
\ec

\item For every $Z_0\in\tn{U}_\rho\Gamma$ and every open set $A$ contained in a neighbourhood of $Z_0$ with local product structure, the map

\bc
$W^{cu}_{\tn{loc}}(Z_0) \too\rr: \hspace{0,3cm} Z\mapsto \nu^{ss}_\tn{loc}(Z)(A\cap W^{ss}_\tn{loc}(Z)) $
\ec

\noindent is upper semi continuous.

\end{itemize}

Consider the family $\lbrace\nu^{cu}_{\tn{loc}}(Z_0)\rbrace_{Z_0\in\tn{U}_\rho\Gamma}$ provided by Fact \ref{fact existence of transverse measure}. Then the following holds:

\begin{itemize}
\item If $A$ is an open set contained in a neighbourhood of $Z_0\in\tn{U}_\rho\Gamma$ with local product structure, set

\bc
$\nu(A):=\displaystyle\int_{Z\in W^{cu}_{\tn{loc}}(Z_0)}\nu^{ss}_{\tn{loc}}(Z)(A\cap W^{ss}_{\tn{loc}}(Z))d\nu^{cu}_{\tn{loc}}(Z_0)(Z)$.
\ec

\noindent Then this measure extends to a finite Borel measure $\nu$ on $\tn{U}_\rho\Gamma$ such that for every $t\in\rr$ the following holds:

\bc
$(\phi_t)_*\nu=e^{(h^s-h^u)t}\nu$.
\ec

\noindent Evaluating the previous equality on $\tn{U}_\rho\Gamma$, we obtain that $h^s=h^u$ and that $\nu $ is $\phi_t$-invariant.

\item The number $h^u$ equals the topological entropy $h$ of the flow and the probability proportional to $\nu$ is the Bowen-Margulis probability of $\phi_t$.

\end{itemize}
\begin{flushright}
$\diamond$
\end{flushright}
\end{hecho}

\subsection{The flows on $\tn{U}_o\Gamma$ and $\tn{U}_\tau\Gamma$}\label{subsec flows uogamma y utaugamma}

\subsubsection{\tn{\textbf{Explicit conjugations between $\tn{U}_\rho\Gamma$, $\tn{U}_o\Gamma$ and $\tn{U}_\tau\Gamma$}}}\label{subsub explicit conjugation}

Recall that $\psi_t=\psi_t^o$ is the flow on $\tn{U}_o\Gamma$ induced by the translation flow (\ref{eq translation flow}).

The following lemma implies in particular that the action of $\Gamma$ on $\bgc\times \rr$ via $c_o$ is proper and co-compact.

\begin{lema}\label{lema conj urhogamma y uogamma}
There exists a Hölder homeomorphism $\tn{U}_\rho\Gamma \too \tn{U}_o\Gamma$ that conjugates the flows $\phi_t$ and $\psi_t$. Further, for every point $(x_0,y_0,t_0)\in\tn{U}_o\Gamma$ the central unstable and strong stable leaves through $(x_0,y_0,t_0)$ are given by

\bc
$W^{cu}(x_0,y_0,t_0)=\lbrace (x_0,y,t)\in\tn{U}_o\Gamma:\hspace{0,3cm} y\in\bg\setminus \Gamma\cdot x_0 \tn{ and } t\in\rr \rbrace$
\ec

\noindent and

\bc
$W^{ss}(x_0,y_0,t_0)=\lbrace (x,y_0,t_0)\in\tn{U}_o\Gamma:\hspace{0,3cm} x\in\bg\setminus \Gamma\cdot y_0\rbrace$.
\ec

\end{lema}

\begin{proof}

Consider the map $F_\rho\too\bgc\times\rr$ defined by 

\bc
$(x,y,\theta,v)\mapsto \left(x,y,-\frac{1}{2}\log\vert\langle v,J^o\cdot v\rangle_{p,q}\vert\right)$,
\ec

\noindent which is easily seen to be Hölder continuous, injective and equivariant. Moreover one can prove that it is proper and surjective, hence a homeomorphism.

The statement involving the flows and the laminations is straightforward.

\end{proof}

We now turn our attention to the translation flow on $\tn{U}_\tau\Gamma$. An analogue of Lemma \ref{lema conj urhogamma y uogamma} is also available. In fact, the analogue holds because of the following remark.

\begin{rem}\label{rem co and ctau are cohomologous}

The cocycles $c_o$ and $c_\tau$ are cohomologous. Indeed, this follows from the fact that $c_o$ and $c_\tau$ have the same periods and a theorem due to Livsic \cite{Liv}. Explicitly, let

\bc
$U:\bg\too\rr: \hspace{0,3cm} U(x):=\dfrac{1}{2}\log\dfrac{\Vert v_x\Vert_\tau\Vert \theta_x\Vert_\tau}{\vert \theta_x (J^o\cdot v_x)\vert}$.
\ec

\noindent Then for every $\gamma$ in $\Gamma$ and $x$ in $\bg$ one has

\bc
$c_\tau(\gamma,x)-c_o(\gamma,x)=U(\gamma\cdot x)-U(x)$ .
\ec
\begin{flushright}
$\diamond$
\end{flushright}
\end{rem}

\subsubsection{\tn{\textbf{Patterson-Sullivan probabilities for $c_o$ and $c_\tau$}}}\label{subsub PS}

The goal of this subsection is to show the existence of a \textit{Patterson-Sullivan probability of dimension $h^u$} for the cocycle $c_o$, that is, a probability measure $\mu_o$ on $\bg$ such that
\begin{equation}\label{eq PS for co in appendix}
\dfrac{d\gamma_*\mu_o}{d\mu_o}(x)=e^{-h^uc_{o}(\gamma^{-1},x)}
\end{equation}
\noindent holds for every $\gamma\in\Gamma$\footnote{The constant $h^u$ is the one introduced in Fact \ref{fact existence of transverse measure}.}. We will see in the next subsection that in fact one has $h^u=h$.  The existence of a Patterson-Sullivan probability $\mu_\tau$ for $c_\tau$ follows directly from this one by Remark \ref{rem co and ctau are cohomologous}.

When $\Gamma$ is the fundamental group of a closed negatively curved manifold, the existence (and uniqueness) of such a probability is proved by Ledrappier \cite{Led}. When $\rho(\Gamma)$ is Zariski dense one can apply the work of Quint \cite{Qui2} and for the case of $\hpq$-convex co-compact groups we find also the construction presented by Glorieux-Monclair \cite{GM}.

Even though Patterson's method \cite{Pat} works correctly in our setting and produces directly a Patterson-Sullivan probability of dimension $h$, we choose a shorter approach. Applying Fact \ref{fact existence of transverse measure} and Lemma \ref{lema conj urhogamma y uogamma} we find an invariant measure $\lbrace\nu^{cu}_\tn{loc}(u_0)\rbrace_{u_0\in\tn{U}_o\Gamma}$ for the strong stable lamination of $\psi_t:\tn{U}_o\Gamma \circlearrowleft$ which has the property of being contracted by the flow. Lifting this measure to $\bgc\times\rr$ yields a probability $\mu_o$ on $\bg$ satisfying (\ref{eq PS for co in appendix}). Indeed, for closed negatively curved manifolds and the Busemann cocycle this procedure is explained for instance by Babillot in \cite[Subsection 7.1]{Ba}. With obvious adaptations the procedure equally applies in our setting.

\begin{rem}\label{rem hu positive}
Recall that $h^u\geq 0$. Equality (\ref{eq PS for co in appendix}) shows in fact that $h^u$ is positive. Otherwise the probability $\mu_o$ would be $\Gamma$-invariant but one can see that this is not possible for a non elementary word hyperbolic group.
\begin{flushright}
$\diamond$
\end{flushright}
\end{rem}

We finish this subsection by showing that $\mu_o$ has no atoms (this property is needed in the proof of Proposition \ref{prop distribution on bg for length}). The proof presented here is an adaptation of \cite[Proposition 4.3]{GM}.

\begin{lema}\label{lema PS non atomic}
The measure $\mu_o$ has no atoms.
\end{lema}

\begin{proof}

Suppose that there exists an atom $y\in\bg$ for $\mu_o$. Since $h^u$ is positive the point $y$ cannot be fixed by an element of $\Gamma$, hence
\begin{equation}\label{eq en lema PS non atomic}
1=\mu_o(\bg)\geq \displaystyle\sum_{\gamma\in\Gamma}e^{-h^uc_o(\gamma^{-1},y)}\mu_o(y).
\end{equation}
\begin{cla} \label{claim on lema PS non atomic}

There exists a sequence $\gamma_n\too\infty$ such that $c_o(\gamma_n^{-1},y)\too -\infty$. 

\end{cla}

\begin{proof}[Proof of Claim \ref{claim on lema PS non atomic}]
Let $x$ be a point in $\bg$ different from $y$ and $\Vert\cdot\Vert$ be any norm on $\rr^d$. Take a sequence $\gamma_n\too\infty$ such that 

\bc
$(\gamma_n)_+\too y$ and $(\gamma_n)_-\too x$.
\ec

\noindent By taking a subsequence if necessary we may suppose that $\gamma_n$ converges uniformly to $y$ on compact sets of $\bg\setminus\lbrace x\rbrace$ (c.f. Bowditch \cite[Lemma 2.11]{Bowd}). Let $B(x)\subset\bg$ be the complement of a small neighbourhood of $x$ in $\bg$ and $b(y)\subset B(x)$ be a small neighbourhood of $y$. Then we can suppose that the inclusion $\gamma_n\cdot B(x)\subset b(y)$ holds for every $n$.

By Proposition \ref{prop limit with S and U and Uuno cerca gammamas} there exists $\varepsilon >0$ such that for all $n$ one has

\bc
$\xi(B(x))\subset B_\varepsilon(S_{d-1}(\rho(\gamma_n)))$. 
\ec

\noindent Take a positive $c$ with the following property: for every $n$ and every vector $v$ in the set $B_\varepsilon (S_{d-1}(\rho(\gamma_n)))$ one has

\bc
$\Vert\rho(\gamma_n)\cdot v\Vert\geq c\Vert\rho(\gamma_n)\Vert\Vert v\Vert$.
\ec

Let $v\neq 0$ be a vector in $\xi(y)$. We have that $\rho(\gamma_n^{-1})\cdot v$ belongs to $B_\varepsilon(S_{d-1}(\rho(\gamma_n)))$ hence

\bc
$\rho(\gamma_n^{-1})\cdot v\too 0$
\ec

\noindent as $n\too\infty$. The divergence $c_o(\gamma_n^{-1},y)\too -\infty$ follows.

\end{proof}

A combination of (\ref{eq en lema PS non atomic}) and Claim \ref{claim on lema PS non atomic} yields the desired contradiction.

\end{proof}

\subsubsection{\tn{\textbf{The Bowen-Margulis probability on $\tn{U}_o\Gamma$ and $\tn{U}_\tau\Gamma$}}}\label{subsub BM for uogamma and utaugamma}

Recall that $[\cdot,\cdot]_o$ is the Gromov product of the pair $\lbrace c_o,c_o\rbrace$.

\begin{prop}[Sambarino {\cite[Theorem 3.2]{Sam}}]\label{prop product is of maximal entropy}
The number $h^u$ equals the topological entropy $h$ of $\psi_t$ and the measure

\bc
$e^{-h[\cdot,\cdot]_o}\mu_o\otimes\mu_o\otimes dt$
\ec

\noindent induces a measure on the quotient space $\tn{U}_o\Gamma$ proportional to the Bowen-Margulis probability of $\psi_t$.

\end{prop}

\begin{proof}
From explicit computations one can show that $e^{-h^u[\cdot,\cdot]_o}\mu_o\otimes\mu_o\otimes dt$ equals the product of measures $\nu_\tn{loc}^{cu}$ and $\nu^{ss}_\tn{loc}$ as in Fact \ref{fact product measure is BM}.
\end{proof}

We now state the desired result of this appendix: the analogue of \cite[Proposition 4.3]{Sam}. Provided with Proposition \ref{prop product is of maximal entropy}, the same proof applies in our setting.

\begin{prop}[Sambarino {\cite[Proposition 4.3]{Sam}}]\label{prop distribution of periodic orbits}

There exists a positive $M=M_{\rho,o}$ such that

\vspace{0,2cm}

\bc
$M e^{-ht}\displaystyle\sum_{\gamma\in\gh, \ell_{c_o}(\gamma)\leq t} \delta_{\gamma_-}\otimes\delta_{\gamma_+}\too e^{-h[\cdot,\cdot]_o}\mu_o\otimes\mu_o$
\ec

\noindent as $t\too\infty$ on $C_c^*(\bgc)$.
\begin{flushright}
$\square$
\end{flushright}
\end{prop}

For the flow on $\tn{U}_\tau\Gamma$ we obtain analogue results.

\begin{rem}\label{rem BM for ctau and distribution}
Let $[\cdot,\cdot]_\tau$ be the Gromov product of the pair $\lbrace c_o,c_\tau\rbrace$ defined in Subsection \ref{subsec dual and gromov ctau}. The same arguments of Proposition \ref{prop product is of maximal entropy} and Proposition \ref{prop distribution of periodic orbits} apply to obtain that

\bc
$e^{-h[\cdot,\cdot]_\tau}\mu_o\otimes\mu_\tau\otimes dt$
\ec

\noindent induces the Bowen-Margulis probability of the translation flow on $\tn{U}_\tau\Gamma$ and that there exists a positive $M'=M'_{\rho,\tau}$ such that 

\bc
$M'  e^{-ht}\displaystyle\sum_{\gamma\in\gh, \ell_{c_\tau}(\gamma)\leq t} \delta_{\gamma_-}\otimes\delta_{\gamma_+}\too e^{-h[\cdot,\cdot]_\tau}\mu_o\otimes\mu_\tau$
\ec

\noindent as $t\too\infty$ on $C_c^*(\bgc)$.
\begin{flushright}
$\diamond$
\end{flushright}
\end{rem}

\renewcommand{\bibname}{Referencias}

\thispagestyle{empty}

\end{document}